\documentclass[a4paper,reqno]{amsart}
\usepackage{array}
\usepackage{siunitx}
\usepackage{enumerate,amsmath,amsthm,amsfonts}
\usepackage{subfigure}
\usepackage{hyperref}
\usepackage{url}
\usepackage{longtable}
\usepackage{amssymb}
\usepackage{braket,enumerate,bm}
\usepackage{graphicx,color}
\usepackage{graphpap}
\usepackage{latexsym}
\usepackage{float}
\usepackage{setspace}
\usepackage{lineno}   

\newtheorem{thm}{Theorem}[section]
\newtheorem{prop}[thm]{Proposition}
\newtheorem{lem}[thm]{Lemma}

\newtheorem{rem}[thm]{Remark}
\newtheorem{definition}[thm]{Definition}
\newtheorem{Quotation}[thm]{Quotation}
\newtheorem{obs}[thm]{Observation}

\DeclareMathOperator {\Exp}{E}

\DeclareMathOperator {\Le}{L}
\DeclareMathOperator {\K}{K}
\DeclareMathOperator {\rank}{rank}
\DeclareMathOperator {\Tr}{Tr}
\DeclareMathOperator {\Var}{Var}

\def\A{\mathbf{A}}
\def\B{\mathbf{B}}
\def\bR{\mathbf{R}}
\def\cas{\mathrel{\stackrel{a.s.}{\to}}}
\def\comment#1{}
\def\D{\mathbf{D}}
\def\E{\mathbf{E}}

\def\fra#1#2{{#1}/{#2}}
\def\G#1#2{G_{#1,#2}}
\def\iidsim{\mathrel{\stackrel{\mbox{\rm i.\@i.\@d.\@}}{\sim}}}
\def\meas{\mu}
\def\MPmeas#1#2{\mu_{#1,#2}}
\def\N{\mathrm{N}}

\DeclareMathOperator {\Prb}{P}
\newcommand{\R}{\mathbb R}
\def\P{\mathbf{P}}
\def\qtl#1{\left(#1\right)^-}

\def\P{\mathbf{P}}
\def\S{\mathbf{S}}
\def\U{\mathbf{U}}

\def\x{\mathbf{x}}
\def\I{\mathbf{1}}
\def\x{\bm{x}}
\def\X{\mathbf{X}}
\def\Y{\mathbf{Y}}
\def\M{\mathbf{M}}
\def\Rc{\mathbf{C}}

\font\rs=cmss10.360pk
\font\rt=cmss9.360pk
\font\sd=cmcsc9.360pk

\include {mak}
\begin{document}
~\vspace{-16mm}

\oddsidemargin 16.5truemm
\evensidemargin 16.5truemm

\thispagestyle{plain}

\hspace{8.4cm}{\rs J. Indones. Math. Soc.}

\vspace{-0.25cc}

\hspace{7.5cm}{\scriptsize Vol. xx, No. xx (20xx), pp.~xx--xx.}\\
\vspace{1.2cc}

\vspace{1.5cc}

\begin{center}
{\Large\bf A DICHOTOMOUS BEHAVIOR OF GUTTMAN-KAISER CRITERION FROM EQUI-CORRELATED NORMAL POPULATION\\ 
\vspace{1.5cc}
{\large\sc Yohji Akama$^{1}$, Atina Husnaqilati$^{1,2}$}\\

\vspace{0.3 cm} {\small $^{1}$Department of Mathematics, Graduate School of Science, Tohoku
University, Aramaki, Aoba, Sendai, 980-8578, Japan, yoji.akama.e8@tohoku.ac.jp\\ 
$^{2}$Department of Mathematics, Faculty of Mathematics and Natural Science, Universitas Gadjah Mada, Sekip Utara Bulaksumur, Yogyakarta 55281, Indonesia, husqila@gmail.com}

\rule{0mm}{6mm}\renewcommand{\thefootnote}{}\footnotetext{
\\
{\rm Received: dd-mm-yyyy, accepted: dd-mm-yyyy.}
}
}

\vspace{1.5cc}

\parbox{24cc}{{\Small{\bf Abstract.}
We consider a $p$-dimensional, centered normal population such that all variables have a positive variance $\sigma^2$ and any correlation coefficient between different variables is a given nonnegative constant $\rho<1$. Suppose that both the sample size $n$ and population dimension $p$ tend to infinity with $p/n \to c>0$. 
We prove that the limiting spectral distribution of a sample correlation matrix is Mar\v{c}enko-Pastur distribution of index $c$ and scale parameter $1-\rho$.
By the limiting spectral distributions, we rigorously show the limiting behavior of widespread stopping rules Guttman-Kaiser criterion and cumulative-percentage-of-variation rule in PCA and EFA.
As a result, we establish the following dichotomous behavior of Guttman-Kaiser criterion when both $n$ and $p$ are large, but $p/n$ is small: (1) the criterion retains a small number of variables for $\rho>0$, as suggested by Kaiser, Humphreys, and Tucker [Kaiser, H. F. (1992). On Cliff's formula, the Kaiser-Guttman rule and the number of factors. \emph{Percept. Mot. Ski.} 74]; and
(2) the criterion retains $p/2$ variables for $\rho=0$, as in a simulation study [Yeomans, K. A. and Golder, P. A. (1982). The Guttman-Kaiser criterion as a predictor of the number of common factors.  \emph{J. Royal Stat. Soc. Series D.} 31(3)].
}}
\end{center}

\vspace{0.25cc}
\parbox{24cc}{\Small {\it Key words and Phrases}: Guttman-Kaiser criterion; equi-correlated normal population; Mar\v{c}enko-Pastur distribution; sample correlation matrices; limiting spectral distribution.
}

\vspace{1.5cc}

\vspace{1.5cc}

\section{Introduction} 
For large datasets, it is necessary to reduce the dimensionality  from $n$ observations on $p$ variables. Numerous techniques have been developed, such as principal component analysis (PCA)~\cite{jackson,J} and exploratory factor analysis (EFA)~\cite{harman,mulaik,lean}
for this goal. 
PCA and EFA discover relationships among a set of potentially associated variables and merge them into smaller groups called as `principal components' (in PCA) or `factors' (in EFA)~\cite{jackson,mulaik}. The number of non-trivial principal components or factors is usually suggested by a \emph{stopping rule}~\cite[p.~41]{jackson}.  

A well-known stopping rule in PCA and EFA is \emph{Guttman-Kaiser criterion}~\cite{guttman,ka1,zwick2}.
This rule may be the most widely used rule to retain principal components and factors~\cite{lean} because of its clearness, ease of implementation, and default stopping rule in statistical software such as SPSS and SAS. However, nearly four decades ago, for independent normal random variables, Yeomans-Golder~\cite{Criterion} observed that Guttman-Kaiser criterion retains a large number of factors at most $p/2$ by a simulation study. Moreover, for dependent variables, H.~F.~Kaiser, who introduced Guttman-Kaiser criterion, adverted to a dichotomous behavior of the criterion by reporting the following experience of experts of EFA:
\begin{Quotation}[\protect{\cite{kai}}]\label{quotation}\rm ...
Humphreys (personal communication, 1984) asserts that, when the number $p$ of attributes is large and the ``average'' intercorrelation is small, the Kaiser-Guttman rule will overfactor. Tucker (personal communication, 1984) asserts that, when the number of attributes $p$ is small and the structure of the attributes is particularly clear, the Kaiser-Guttman rule will underfactor. ...
\end{Quotation}
\noindent
Here, ``overfactor''~(``underfactor'', resp.) means ``overestimate''~(``underestimate'', resp.) the number of factors in the factor model. According to Kaiser \cite{kai2}, `the ``average'' intercorrelation' corresponds to a positive constant $\rho$ of the following correlation matrix
\begin{align*}
\Rc(\rho)=\begin{bmatrix}
1 & \rho & \cdots & \rho\\
\rho & 1 & \cdots & \rho\\
\vdots & \vdots & & \vdots\\
\rho &\rho &\cdots & 1
\end{bmatrix}\in \mathbb{R}^{p\times p}.
\end{align*}

In statistics, Kaiser's observation in Quotation~\ref{quotation} and the simulation study of Yeomans-Golder~\cite{Criterion} showed that Guttman-Kaiser criterion has a considerable risk of overfactor and underfactor. 
Consequently, applying this criterion may significantly impact the interpretation in PCA and EFA~\cite{Criterion,zwick2}. However, no systematic analysis of Guttman-Kaiser criterion has been done for datasets that have both large $n$ and $p$ large. Because of the wide use of Guttman-Kaiser criterion, it is necessary to verify Kaiser's observation in Quotation~\ref{quotation} and a simulation study of Yeomans-Golder~\cite{Criterion}. Thus, we study the behavior of Guttman-Kaiser criterion to detect the number of components or factors regarding $p/n$ and $\rho$ (the correlation coefficient between variables).

In mathematics, the most aforementioned works assume that the $p$ entries of variables are independent to show the behavior of the largest entries of the sample covariance matrices and the sample correlation matrices. Fan-Jiang~\cite{fan} considered $p$-dimensional population such that all the $p$ variables obey normal distributions where all the correlation coefficients between different variables are $\rho$ $(0\le \rho<1)$. We call this population an \emph{equi-correlated normal population} (ENP). Here, we consider how an ENP impacts the \emph{limiting spectral distribution} (LSD) of the sample covariance and the sample correlation matrices.

In this paper, we  expound on the behavior of Guttman-Kaiser criterion regarding $p/n$ and $\rho$ as $n,p\to\infty$ by verifying the impact of $\rho$ on the LSDs of the sample covariance matrices and the sample correlation matrices. 
First of all, we precisely compute and illustrate $q/p$ in $n,p\to\infty$, $p/n\to c>0$ by the limit of \emph{empirical spectral distribution} (ESD) from random matrix theory~\cite[p. 5]{bai4}. Here, $q$ is the number of principal components or factors that Guttman-Kaiser criterion preserves. For a real symmetric matrix $\M$ of order $p$, the ESD of $\M$ is, by definition, a function
 \begin{align*}
   F^\M(x)=\dfrac{1}{p}\#\Set{1\le i\le p|\lambda_i\le x}
 \end{align*}
  where $\lambda_1\ge\lambda_2\ge\cdots\ge\lambda_p$ are the eigenvalues of $\M$.
Mathematically speaking, the proportion of eigenvalues rejected by Guttman-Kaiser criterion is $F^\S(p^{-1}\Tr \S)$ or the sample correlation matrix $F^\bR(1)$ where $\S$ denotes sample covariance matrix and $\bR$ denotes sample correlation matrix. Here, the LSD is the limit of ESD in $n,p\to\infty,\ p/n\to c>0$. 
From an ENP, we prove that 
\begin{enumerate}
    \item if the population mean and variance are, moreover, $\bm {0}$ and $\sigma^2$, then the LSD of the sample covariance matrix $\S$ is the \emph{Mar\v{c}enko–Pastur distribution}~\cite{yao} with index $c$ and  scale parameter by $\sigma^2(1-\rho)$; and
\item  the LSD of sample correlation matrix is the Mar\v{c}enko-Pastur distribution with index $c$ and scale parameter $1-\rho$.
(This answers a question posed in Fan-Jiang~\cite[Remark 2.5]{fan}).
\end{enumerate}

Thus, by properties of Mar\v{c}enko–Pastur distribution, we establish the following \emph{dichotomous behavior} of Guttman-Kaiser criterion when both
of $n$ and $p$ are large, but $p/n$ is small: (1) the criterion retains a small number of
variables for $\rho> 0$, as suggested by Quotation~\ref{quotation}; and (2) the criterion retains $p/2$ variables for $\rho= 0$, as suggested by Yeomans-Golder~\cite{Criterion}.

\subsection*{The organization of this paper}  In Section~\ref{sec:gk and jr}, from an ENP with $0\le \rho<1$, we show that (1)
if furthermore the population is centered with positive variance $\sigma^2$, then the LSD of the sample covariance matrices is Mar\v{c}enko-Pastur distribution scaled by $\sigma^2(1-\rho)$ and that (2) the LSD of the sample correlation matrices is Mar\v{c}enko-Pastur distribution scaled by $(1-\rho)$.
In Section~\ref{sec:components retention rules}, 
by using Theorems of Section~\ref{sec:gk and jr}, we compute the limits of $q/p$ from Guttman-Kaiser criterion and \emph{cumulative-percentage-of-variation rule} (CPV rule)~~\protect{\cite[p.~113]{J}} in $n,p\to\infty$ with $p/n\to c>0$. The comparison of both limits is also presented. By these results, we illustrate the limiting behavior of Guttman-Kaiser criterion  discuss (1) Kaiser's observation~\cite{kai}, that is, Quotation~\ref{quotation}, and (2) a simulation study of Yeomans-Golder~\cite{Criterion} about how many factors Guttman-Kaiser's criterion retains in EFA. In Section~\ref{sec:app}, we also study Guttman-Kaiser criterion and CPV rule with datasets from economics and molecular biology, and discuss future work. Section~\ref{sec:conclusion} is conclusion.

\section{The spectral analysis of equi-correlated normal population}
\label{sec:gk and jr}
Let $X_1,\ldots,X_n \in\R^{p}$ be a random sample from a $p$-dimensional population.
 For the data matrix $\X=[X_1,\ldots,X_n]\in\R^{p\times n}$
 , write $
 \X=\left[x_{ij}\right]_{p\times n}=\left[\x_{1}^\top,\x_{2}^\top,\ldots,\x_{p}^\top\right]^\top$.
 For $1\le i\le p$, let $\bar{x}_i$ be the sample average of $\x_{i}\in \R^{1\times n}$, that is, $\bar{x}_i=n^{-1}\sum_{j=1}^nx_{ij}$.
 We write $\bar{\x}_i$ for $\bar{x}_i\bm e$, where $\bm{e}=[1,\ldots,1]\in \R^{1\times n}$. 
Let $\bar {\X}$ be $[\bar x_1,\bar x_2,\ldots,\bar x_p]^\top \bm{e}\in\R^{p\times n}$ and $\E$ be $n^{-1/2}(\mathbf{X}-\bar{ \X})$. 
The $p\times p$ sample correlation matrix is defined as
\begin{align*}
\bR&=\Y\Y^\top\ \text{where}\  \Y^\top=\left[\frac{(\x_{1}-\bar\x_1)^\top}{\|\x_{1}-\bar\x_1\|},\ldots,\frac{(\x_{p}-\bar\x_p)^\top}{\|\x_{p}-\bar\x_p\|}\right].
\end{align*}
Here, the notation $\|.\|$ is the Euclidean norm. The noncentered  sample  correlation matrix is,  by  definition, $\tilde{\bR}=\mathbf{\tilde{\Y}}\mathbf{\tilde{\Y}}^\top$ with $\mathbf{\tilde{ Y}}^\top=\left[\x_{1}^\top/\|\x_{1}\|,\ldots,\x_{p}^\top/\|\x_{p}\|\right]$. The $p\times p$ sample covariance matrix is denoted by $\S=n^{-1}\X\X^\top$, and the centered sample  covariance  matrix by $\tilde \S=\E\E^\top$.

By a \emph{distribution function}, we mean a nondecreasing, right-continuous function $F:\R\to[0,1]$ such that $\lim_{x\to-\infty} F(x)=0$ and $\lim_{x\to+\infty} F(x)=1$. 
According to~\cite[p.~10]{yao}, the Mar\v{c}enko-Pastur distribution $F_{c,\sigma^2}$ with index $c>0$ and scale parameter $\sigma^2>0$ has a density function
\begin{align*}
\frac{1}{2\pi c\sigma^2 x}\sqrt{(b_{\sigma^2}(c)-x)(x-a_{\sigma^2}(c))}\I_{x\in[a_{\sigma^2}(c),\,b_{\sigma^2}(c)]}
\end{align*}
and has a point mass of value $1 - 1/c$ at the origin if $c > 1$.
Here, $\I$ denotes the indicator function, $a_{\sigma^2}(c) = \sigma^2(1-\sqrt{c})^2$, and $b_{\sigma^2}(c) = \sigma^2(1 +\sqrt{c})^2.$ The Mar\v{c}enko-Pastur distribution has expectation $\sigma^2$. Moreover, by a random variable transformation, we can write that $F_{c,\sigma^2}(x)=F_{c,1}(x/\sigma^2)$ for all $x\in \R$. 

Let $F$ be a distribution function on $\R$ and let $\left(F_n:n\in\mathbb{N}\right)$ be a sequence of distribution functions. We say that $F_n$ \emph{weakly converges} to a function $F:\R\to\R$ if $F_n(x)$ converges to $F(x)$ at every point $x\in\R$ where $F$ is continuous.

Mar\v{c}enko-Pastur~\cite{Marcenko} first found the LSD of the sample covariance matrices. Their result has since been extended in various directions. The following is widely known.
\begin{prop}[\protect{\cite[p.~12]{yao}}]\label{prop:MP LT} Assume that the entries of $\mathbf{X}=[x_{ij}]_{p\times n}$ are centered \emph{i.\@i.\@d.\@} random variables with variance $\sigma^2$, and $n,p\to\infty$ with $p/n\rightarrow c> 0$. Then, almost surely, $F^\S$ weakly converges to $F_{c,\sigma^2}$. 
\end{prop}
Furthermore, Jiang~\cite{Ji} established the LSD of the sample correlation matrix $\bR$. 
\begin{prop}[\protect{\cite[{Theorem 1.2}]{Ji}}\label{prop:ji3}]
Assume that the entries of $\mathbf{X}=[x_{ij}]_{p\times n}$ are \emph{i.\@i.\@d.\@} random variables. Suppose $\Exp\left|x_{ij}\right|^2<\infty$ and $n,p\to\infty$ with $p/n\to c>0$. Then, almost surely, $F^{\bR}$ weakly converges to $F_{c,1}$. This conclusion is true for $F^{\tilde{\bR}}$ provided $\Exp(x_{ij})=0$ in addition.
\end{prop}
We discuss the LSDs of the sample covariance matrices and the sample correlation matrices formed from an ENP. An ENP is written as  $\mathrm{N}_p(\bm{\mu},\,\mathbf{D}\Rc(\rho)\mathbf{D})$ for a deterministic vector $\bm{\mu}\in\R^p$ and a deterministic nonsingular diagonal matrix $\mathbf{D}\in\R^{p\times p}$. 

We assume $X_1,\ldots,X_n\stackrel{\mbox{\rm i.\@i.\@d.\@}}{\sim} \mathrm{N}_p(\bm{\mu},\,\mathbf{D}\Rc(\rho)\mathbf{D})$,
$0\le \rho< 1$, and $n,p\to \infty$ with $p/n\to c>0$. 
Then, in the following section, as $n,p\to\infty$ with $p/n\to c>0$, we obtain that (a) the LSD of the sample covariance matrix is $F_{c,\sigma^2(1-\rho)}$ for centered random variables and $\mathbf{D}=\sigma\mathbf{I}$ where $\mathbf{I}$ is the identity matrix of order $p$, (b)
the LSD of the sample correlation matrix is $F_{c,1-\rho}$.

\subsection{Sample covariance matrices}
We recall the following propositions from Huber-Ronchetti~\cite{huber}, Bai-Silverstein~\cite{bai4}, and Anderson et al.~\cite{anderson}.
The \emph{L{\'e}vy distance} between two distribution functions $F$ and $G$ is denoted by $\Le(F,G)$. 
By Huber-Ronchetti~\cite[Definition~2.7]{huber}, the $\Le(F,G)$ is defined as 
\begin{align}\label{levy}
    \Le(F,G)=\inf\Set{\varepsilon>0|F(x-\varepsilon)-\varepsilon\le G(x)\le F(x+\varepsilon)+\varepsilon}.
\end{align}
The following proposition means that a distribution function $F$ weakly converges to a distribution function $G$ if and only if  $\Le(F,G)\to 0$.
\begin{prop}[\protect{\cite[Lemma~2.9]{huber}}] \label{huber1}
The L{\'e}vy distance metrizes the weak topology of the set of distribution functions.
\end{prop}
\begin{proof} We just show that the convergence of $F_n\to F$ at continuity points of $F$ and $\Le(F_n,F)\to 0$ are equivalent as $n\to\infty$. (i) Assume $\Le(F_n,F)\to 0$ as $n\to\infty$. Let $x$ be a continuity point of $F$. Then, $F(x\pm \varepsilon)\pm \varepsilon\to F(x)$ as $\varepsilon\to 0$. By this and~\eqref{levy}, we have $F_n(x)\to F(x)$ as $n\to\infty$. (ii) Assume $F_n(x)\to F(x)$ for all continuity points $x$ of $F$. Let $x_i$ ($0\le i\le N$) be strictly decreasing continuity points of $F$ such that $F(x_0)<\varepsilon/2$ and $F(x_N)>1-\varepsilon/2$, and that $x_{i+1}-x_i<\varepsilon$. Let $N\in\mathbb{N}$ be sufficiently large such that $|F_n(x_i)-F(x_i)|\le\varepsilon/2$ for all $n\ge N$. By this and monotonicity of $F_n$ and $F$, for $x\in [x_{i-1},x_i]$, we have $F_n(x)\le F_n(x_i)<F_n(x_i)+\varepsilon/2\le F(x+\varepsilon)+\varepsilon$ for all $n\ge N$. This bound also holds for $x<x_0$ and $x>x_N$. In the same way, we establish $F_n(x)\ge F(x-\varepsilon)-\varepsilon$.
\end{proof}
\begin{figure}
      \includegraphics[scale=0.5]{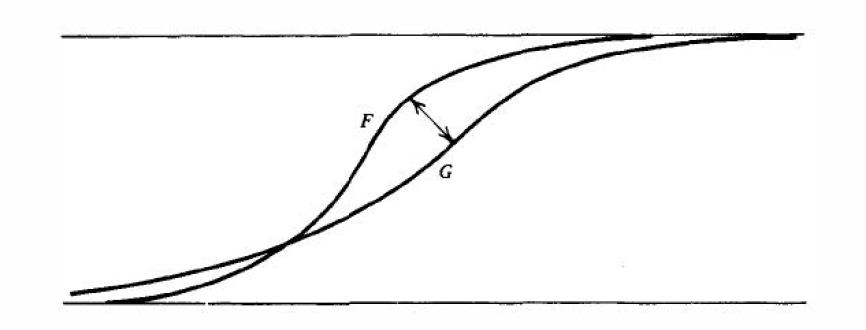}
      \caption{The maximum distance between the graph of {$F$} and {$G$}, measured along a {$45^{\circ}$}-direction ($\sqrt{2}L(F,G)$)~\cite[p.~28]{huber}.}
      \end{figure}

\medskip
The \emph {Kolmogorov distance} between two distribution functions $F$ and $G$ is defined as
\begin{align}\label{kolmogorov}
\K(F,G)=\sup_ {x\in \mathbb{R}}\left|F(x)-G(x)\right|
\end{align}
do not generate the weak topology,  but they possess  other convenient properties.  In particular, we have the following inequalities.
\begin{prop}[\protect{\cite[p.~36]{huber}}]\label{huber2}
For any distribution functions $F$ and $G$,
$\Le(F,G)\le \K(F,G).$
\end{prop}
\begin{proof}
Let $\K(F,G)$ be a positive $a$. Then, by \eqref{kolmogorov}, $F(x)-G(x)\le a$ and $G(x)-F(x)\le a$ for all $x\in\R$. By this, for all $x\in\R$, we can write $F(x-a)-a\le F(x)-a\le G(x)$ and  $G(x)\le F(x)+a\le F(x+a)+a$ since $F$ and $G$ are nondecreasing function. By \eqref{levy}, we have Proposition~\ref{huber2}.
\end{proof}
In cases where the underlying variables are dependent, the following proposition is powerful.
\begin{prop}[\protect{\cite[Lemma~2.6, the rank inequality]{bai4}}] \label{rank}
\begin{align*}
\K\left(F^{\A\A^\top},F^{\B\B^\top}\right)\leq\frac{1}{p}\;\rank(\A-\B),\qquad(\A,\B\in\R^{p\times n}).\end{align*}
\end{prop}

We are concerned with a well-known decomposition of a centered, equi-correlated normal random vector:
\begin{prop}[\protect{\cite[(3.21)]{fan}}]\label{lem:decomposition}
If $X_1,\ldots,X_n\stackrel{\mbox{\rm i.\@i.\@d.\@}}{\sim} \mathrm{N}_p(\mathbf{0},\,\sigma^2\Rc(\rho))$ for $\sigma>0$ and $0\le \rho< 1$, then we can find independent, standard normal random variables $\eta_j,\xi_{ij}$ $(1\le i\le p,\ 1\le j\le n)$ such that \begin{align}X_j=\sigma\sqrt{\rho}\,[\eta_j,\ldots,\eta_j]^\top+\sigma\sqrt{1-\rho}\,[\xi_{1j},\ldots,\xi_{pj}]^\top.\label{xij decomposition}
\end{align}
\end{prop}
Then, Proposition~\ref{prop:MP LT} derives the following:

\begin{thm}[Shrinkage]\label{thm:equilsd}
Suppose $X_1,\ldots,X_n\stackrel{\mbox{\rm i.\@i.\@d.\@}}{\sim} \mathrm{N}_p(\mathbf{0},\,\sigma^2\Rc(\rho))$ for $\sigma>0$, $0\le \rho< 1$, and $n,p\to \infty$ with  $p/n\to c>0$. Then, 
it holds almost surely that both of $F^{\S}$ and $F^{\E\E^\top}$ weakly converge to $F_{c,\sigma^2(1-\rho)}$. 
\end{thm}
\begin{proof}
By Proposition~\ref{lem:decomposition}, the data matrix $\X=[X_1,\ldots,X_n]$ satisfies \begin{align}\X=\P+\U\label{xwx},
\qquad \P =\left[\sigma \sqrt{\rho}\eta_j\right]_{p\times n},\qquad \U=\left[\sigma\sqrt{1-\rho}\xi_{ij}\right]_{p\times n}.
\end{align} 
All entries of the  matrix $U$ are i.\@i.\@d.\@ centered random variables with variance $\sigma^2(1-\rho)$.
Hence, by Proposition~\ref{prop:MP LT},
$F^{\frac{1}{n}\U \U^\top}$ weakly converges to $F_{c,\sigma^2(1-\rho)}$ almost surely,
as $p,n\to\infty,p/n\to c>0$.
Thus, by Proposition~\ref{huber1}, 
\begin{align}\label{xi}
\Le\left(F^{\frac{1}{n}\U \U^\top},\, F_{c,\sigma^2(1-\rho)}\right)\cas0.
\end{align}
Here, the notation $\cas$ means almost sure convergence. Recall that   $\S=n^{-1}\X\X^\top$. 
Thus, by Proposition~\ref{huber2}, Proposition~\ref{rank}, \eqref{xwx}, and $\rank \P\le 1$,
\begin{align*}
\Le\left(F^{\S},\ 
F^{\frac{1}{n}\U \U^\top}\right)\le \K\left(F^{\S},F^{\frac{1}{n}\U \U^\top}\right)
&\le \frac{1}{p}\rank\left(\frac{1}{\sqrt{n}}\X-\frac{1}{\sqrt{n}}\U\right)\\
&=  \frac{1}{p}\rank \left(\frac{1}{\sqrt{n}}\P \right)
\le \frac{1}{p}
\to 0 &(p\to \infty).
\end{align*}
By this and \eqref{xi}, the triangle inequality implies $
\Le\left(F^{\S},\ F_{c,\sigma^2(1-\rho)}\right)\cas 0$.
By Proposition~\ref{huber1}, almost surely, $F^{\S}$ weakly converges to $F_{c,\sigma^2(1-\rho)}$.

By $\E=n^{-1/2}(\mathbf{X}-\bar{ \X})$, Proposition~\ref{huber2}, Proposition~\ref{rank}, and $\rank (\bar {{\X}})\le 1$,
\begin{align*}
\Le(F^{\E\E^\top},F^{\S})\le \K\left(F^{\E\E^\top},F^{\frac{\X\X^\top}{n}}\right)
\le \frac{\rank \left(\bar { {\X}}\right)}{p}
\le \frac{1}{p}\to 0,
\end{align*}
for $p\to \infty$. Thus, 
$\Le\left(F^{\S},\ F_{c,\sigma^2(1-\rho)}\right)\cas 0\iff \Le\left(F^{\E\E^\top},\ F_{c,\sigma^2(1-\rho)}\right)\cas 0.$ 
By Proposition~\ref{huber1}, almost surely, $F^{\E\E^\top}$ weakly converges to $F_{c,\sigma^2(1-\rho)}$. 
\end{proof}
\begin{rem}\label{rem:a}\rm
By using \cite{silver} for \emph{Stieltjes transform},
we can derive this shrinkage theorem for the sample covariance matrix $\Rc(\rho)$ 
 from the property: \emph{$\Rc(\rho)$ has $p-1$ eigenvalues $1-\rho$ and one eigenvalue $1+(p-1)\rho$}. As a feature of our proof, it can release the independence of $X_1,\ldots,X_n$ in Theorem~\ref{thm:equilsd},  so long as the decomposition \eqref{xij decomposition} holds. For example, Theorem~\ref{thm:equilsd} holds for $\mathbf X=[X_1,\ldots,X_n]=[x_{ij}]_{p\times n}$ such that $x_{ij}$ are standard normal random variables equi-correlated with nonnegative $\rho<1$. For further applications of the propositions of this subsection to establish LSDs of various random matrices, see \cite{JIMS-LSD}.
\end{rem}
\subsection{Sample correlation matrices}
In this subsection, we establish that the LSD of $\bR$ is  the LSD of $\S$ for a centered ENP with unit variance and $0\le \rho<1$.
To prove Theorem~\ref{thm:equilsd2}, we need the following lemma.
\begin{lem}\label{lem:lim2} Suppose $X_1,\ldots,X_n\stackrel{\mbox{\rm i.\@i.\@d.\@}}{\sim} \mathrm{N}_p(\mathbf{0},\,\sigma^2\Rc(\rho))$, for $\sigma>0$, $0\le \rho< 1$,
and $n,p\to\infty$ with $p/n\to c>0$. Then, it holds almost surely that $\frac{1}{p}\Tr\S\to \sigma^2$.
\end{lem}

\begin{proof}
By~\eqref{xwx}, $\Tr \S/p$ is equal to
\begin{align}\label{lem:weak}
    \frac{1}{np}\sum_{i=1}^{p}\sum_{j=1}^nx_{ij}^2= \frac{\sigma^2 \rho}{n}\sum_{j=1}^n\eta_j^2+\frac{2\sigma^2\sqrt{\rho(1-\rho)}}{np}\sum_{i=1}^{p}\sum_{j=1}^n\eta_j\xi_{ij}+\frac{\sigma^2(1-\rho)}{np}\sum_{i=1}^{p}\sum_{j=1}^n\xi_{ij}^2.
\end{align}
Since $\eta_j$ and $\xi_{ij}$ are i.\@i.\@d.\@ standard normal random variables for all $1\le i\le p;\ 1\le j\le n$, by the strong law of large numbers, we have $\sum_{j=1}^n\eta_j^2/n\cas 1$, $\sum_{i=1}^{p}\sum_{j=1}^n\eta_j\xi_{ij}/(np)\cas 0$, and $\sum_{i=1}^{p}\sum_{j=1}^n\xi_{ij}^2/(np)\cas 1$. Therefore, almost surely, \eqref{lem:weak} converges to $\sigma^2$.
\end{proof}
The following is Lemma~2 from~\cite{Bai}. It plays a key role in providing the LSD of the sample correlation matrices formed from an ENP with $0\le \rho<1$. 
\begin{prop}[\protect{\cite[Lemma~2]{Bai}}]\label{prop:bai}
Let $\set{x_{ij} |i,j=1,2,\ldots}$ be a double array of \emph{i.\@i.\@d.\@} random variables and let $\alpha>\fra{1}{2}$, $\beta\ge 0$ and $M>0$ be constants. Then,
\begin{align*}
  &  \max_{i\le Mn^\beta}\left|\sum_{j=1}^n\frac{x_{ij}-m}{n^\alpha}\right|\cas 0\ (n\to\infty)\\
&\iff
\Exp|x_{11}|^{\frac{1+\beta}{\alpha}}< \infty\ \&
\  m=\begin{cases}\displaystyle
\Exp x_{11},&(\alpha\le 1),\\
\displaystyle\text{\rm any},&(\alpha>1).
\end{cases}
\end{align*}
\end{prop}

\begin{thm}\label{thm:equilsd2}
Suppose $X_1,\ldots,X_n\stackrel{\mbox{\rm i.\@i.\@d.\@}}{\sim} \mathrm{N}_p(\bm{\mu},\,\mathbf{D}\Rc(\rho)\mathbf{D})$ for a deterministic vector $\bm{\mu}\in\R^p$, a deterministic nonsingular diagonal matrix $\mathbf{D}\in\R^{p\times p}$, and
$0\le \rho< 1$. Suppose $n,p\to \infty$ with $p/n\to c>0$. Then, it holds almost surely that $F^{\bR}$ weakly converges to $F_{c,1-\rho}$. 
\end{thm}
\begin{proof}
Because the sample correlation matrices $\bR$ is invariant under scaling  of variables, we can assume $\mathbf{D}=\mathbf{I}$ without loss of generality. Moreover, we can safely assume that $\bm \mu=\bm 0$ to prove that $F^{\bR}$ weakly converges to $F_{c,1-\rho}$ almost surely, because $\bR$ is invariant under sifting.

To prove that $F^{\bR}$ weakly converges to $F_{c,1-\rho}$ almost surely, it suffices to show that $\Le\left(F^{\bR},\ F^{\E\E^\top}\right)\cas0$, because Theorem~\ref{thm:equilsd} implies  $\Le\left(F^{\E\E^\top},\, F_{c,1-\rho}\right)\cas0$.
By $\bR=\Y\Y^\top$, \cite[Lemma~2.7]{bai4} gives an upper bound on the fourth power of the L\'evy distance 
\begin{align}
\label{levy distance:Cp and EET}
\Le^4\left(F^{\bR},\  F^{\E\E^\top}\right)\le \frac{2}{p}\Tr\left(\Y\Y^\top+\E\E^\top\right)\times 
 \frac{1}{p}\Tr\left(\left(\Y - \E\right)\left(\Y - \E\right)^\top\right).
\end{align}
In the first factor of the right side, $\Tr(\Y\Y^\top)/p=\Tr\bR/p=1$. By $\E=n^{-1/2}(\X -\bar\X)$,
\begin{align}\label{levi2c}
\frac{\Tr(\E\E^\top)}{p}=\frac{1}{np}\sum_{i=1}^{p}\sum_{j=1}^n\left(x_{ij}-\bar x_i\right)^2=\frac{1}{np}\left(\sum_{i=1}^{p}\sum_{j=1}^nx_{ij}^2-n\sum_{i=1}^{p}\bar x_i^2\right)
\le\sum_{i=1}^{p}\sum_{j=1}^n\frac{x_{ij}^2}{np}.
\end{align}
Since the last term almost surely converges to a finite deterministic value
by Lemma~\ref{lem:lim2}, $\limsup |\Tr(\E\E^\top)/p|$ is almost surely finite. 
Therefore, it is sufficient to verify that the second factor of the right side of \eqref{levy distance:Cp and EET} converges almost surely to 0:
\begin{align}
\label{lev}\frac{1}{p}\Tr\left(\left(\Y-\E\right)\left(\Y - \E\right)^\top\right)\cas 0.
\end{align}

Similarly, to prove that $F^{\tilde \bR}$ weakly converges to $F_{c,1-\rho}$ almost surely under $\bm{\mu}=0$, it also suffices to confirm
\begin{align}
\label{levi}
\frac{1}{p}\Tr\left(\left(\mathbf{\tilde{\Y}}-n^{-1/2}\X\right)\left(\mathbf{\tilde{\Y}}-n^{-1/2}\X \right)^\top\right)\cas 0.
\end{align}

\medskip

(a) \textit{The proof of~\eqref{levi}}.
The left side of~\eqref{levi} is $\tilde d_1-2\tilde d_2$ where
\begin{align}\label{levi2a}
\tilde d_1=\frac{1}{np}\sum_{i=1}^{p}\sum_{j=1}^nx_{ij}^2-1,\quad \tilde d_2=\frac{1}{p}\sum_{i=1}^p\left(\frac{\|\x_{i}\|}{\sqrt{n}}-1\right),
\end{align}
because $\Tr(\mathbf{\tilde{\Y}} \mathbf{\tilde{\Y}}^\top)=\Tr(\tilde \bR)=p$ and 
$ \Tr(\X \mathbf{\tilde{\Y}}^\top)=\Tr\left(\left[ \X\frac{\x^\top_{1}}{\|\x_{1}\|},\ldots,\X\frac{\x^\top_{p}}{\|\x_{p}\|}\right]\right)=\sum_{i=1}^p \fra{\x_{i}\x_{i}^\top}{\|\x_{i}\|}=\sum_{i=1}^p \|\x_{i}\|.$
Because $\bf \mu=0$, Lemma~\ref{lem:lim2} implies $\tilde d_1\cas0$. On the other hand, we  prove
\begin{align}\label{suffices}
\tilde d_2=\frac{1}{p}\sum_{i=1}^p \frac{\|\x_{i}\|}{\sqrt{n}}-1=\frac{1}{p}\sum_{i=1}^p\sqrt{\frac{1}{n}\sum_{j=1}^nx_{ij}^2}-1\cas 0.
\end{align}
By the concavity of $\sqrt{x}$ and Lemma~\ref{lem:lim2},
$
 \sqrt{{\sum_{i=1}^p\sum_{j=1}^nx_{ij}^2}/(np)}-1\cas 0
$, so $\limsup_{\substack{n,p\to\infty\\ p/n\to c}}\tilde d_2\ge0$.
Therefore, we have only to demonstrate
    $\liminf_{\substack{n,p\to\infty\\ p/n\to c}}\tilde d_2\ge0$.

By~\eqref{suffices}, it suffices to show
\begin{align}\label{suffices2}
    \liminf_{\substack{n,p\to\infty\\ p/n\to c}}\frac{1}{n}\sum_{j=1}^nx_{ij}^2\ge 1.
\end{align}
By the decomposition \eqref{xij decomposition},
\begin{align}&\frac{1}{n}\sum_{j=1}^nx_{ij}^2
=
 \rho \frac{1}{n}\sum_{j=1}^n\eta_j^2+2\sqrt{\rho(1-\rho)}\frac{1}{n}\sum_{j=1}^n\eta_j\xi_{ij}+(1-\rho)\frac{1}{n}\sum_{j=1}^n\xi_{ij}^2.
 \label{j average}
\end{align}
Following the strong law of large numbers with $\eta_{j}\stackrel{\mbox{\rm i.\@i.\@d.\@}}{\sim}\mathrm{N}(0,1
)$, the first term 
\begin{align}\label{lls}
\frac{1}{n}\sum_{j=1}^n\eta_j^2\cas \Exp\left(\eta_1^2\right)=1,
\end{align}
as $n\to\infty$. The sum of the second and the third terms of the right side of~\eqref{j average} is at least
\begin{align}
  2\sqrt{\rho(1-\rho)}\min_{1\le i\le p}\frac{1}{n}\sum_{j=1}^n\eta_j\xi_{ij}+(1-\rho)\min_{1\le i\le p}\frac{1}{n}\sum_{j=1}^n\xi_{ij}^2\label{unc}. 
\end{align}

By Proposition~\ref{prop:bai} with $\alpha=\beta=1$, $\Exp(\eta_1\xi_{11})=0$, and $\Exp(\xi_{11}^2)=1$, as $n\to\infty$,
\begin{align}
   &\label{x1x2b}\left|\min_{1\le i\le Mn}\frac{1}{n}\sum_{j=1}^n\eta_j\xi_{ij}\right|\le\max_{1\le i\le Mn}\left|\frac{1}{n}\sum_{j=1}^n\eta_j\xi_{ij}\right|\cas 0\\
   &\label{x1x2d}\left|\min_{1\le i\le Mn}\frac{1}{n}\sum_{j=1}^n\xi_{ij}^2-1\right|\le\max_{1\le i\le Mn}\left|\frac{1}{n}\sum_{j=1}^n\xi_{ij}^2-1\right|\cas 0.
\end{align}
By~\eqref{x1x2b}, almost surely, there exists $N_0(m,M)\in \mathbb{N}$ such that
\begin{align}\label{1/m1}
\left|\min_{1\le i\le Mn}\frac{1}{n}\sum_{j=1}^n\eta_j\xi_{ij}\right|\le \frac{1}{m}
\end{align}
for all $m>0$ and for all $n\ge N_0(m,M)$. Since $n,p\to\infty$ with $p/n\to c>0$, there exists $N_1(m)$ such that $(c-1/m)n<p<(c+1/m)n$. Therefore, by~\eqref{1/m1}, for all $n\ge\max(N_1(m),N_0(m,(c-1/m)))$,
\begin{align*}
    -\frac{1}{m}<\min_{1\le i\le \left(c+\frac{1}{m}\right)n}\frac{1}{n}\sum_{j=1}^n\eta_j\xi_{ij}\le \min_{1\le i\le p}\frac{1}{n}\sum_{j=1}^n\eta_j\xi_{ij}\le \min_{1\le i\le \left(c-\frac{1}{m}\right)n}\frac{1}{n}\sum_{j=1}^n\eta_j\xi_{ij}<\frac{1}{m}.
\end{align*}
As a result, for $n\to\infty$, $\min_{1\le i\le p}n^{-1}\sum_{j=1}^n\eta_j\xi_{ij}\cas 0$. 

By~\eqref{x1x2d}, almost surely, there exists $N_2(m,M)$ such that 
\begin{align*}
\left|\min_{1\le i\le Mn}\frac{1}{n}\sum_{j=1}^n\xi_{ij}^2-1\right|\le \frac{1}{m}
\end{align*}
for all $m>0$ and for all $n\ge N_1(m,M)$. Hence, for all $n\ge\max(N_1(m),N_2(m,(c-1/m)))$,
\begin{align*}
    -\frac{1}{m}<\min_{1\le i\le \left(c+\frac{1}{m}\right)n}\sum_{j=1}^n\frac{\xi_{ij}^2-1}{n}\le \min_{1\le i\le p}\sum_{j=1}^n\frac{\xi_{ij}^2-1}{n}
    \le \min_{1\le i\le \left(c-\frac{1}{m}\right)n}\sum_{j=1}^n\frac{\xi_{ij}^2-1}{n}<\frac{1}{m}.
\end{align*}
Therefore, as $n\to\infty$, $\min_{1\le i\le p}n^{-1}\sum_{j=1}^n\xi_{ij}^2\cas 1$.  
As a result, $\eqref{j average}\cas 1-\rho$. By this and \eqref{lls}, $\eqref{suffices2}$ follows. 
Thus, $\tilde d_2\cas0$. Therefore, \eqref{levi} is hold.

\bigskip

(b) \textit{The proof of~\eqref{lev}}. First, the left side of~\eqref{lev} is ${d}_1-2{d}_2$ where
\begin{align}\label{levi2b}
d_1=\frac{1}{np}\sum_{i=1}^{p}\sum_{j=1}^n\left(x_{ij}-\bar x_i\right)^2-1,\quad d_2=\frac{1}{p}\sum_{i=1}^p\left(\frac{\|\x_{i}-\bar \x_i\|}{\sqrt{n}}-1\right),
\end{align}
because $\Tr(\Y\Y^\top)=\Tr\bR=p$ and
\begin{align*}
    \Tr(\E\Y^\top)&=\Tr\left(\left[ \E\frac{(\x_{1}-\bar \x_1)^\top}{\|\x_{1}-\bar \x_1\|},\ldots,\E\frac{(\x_{p}-\bar \x_p)^\top}{\|\x_{p}-\bar \x_p\|}\right]\right)\\
    &=\sum_{i=1}^p \frac{(\x_{i}-\bar \x_i)(\x_{i}-\bar \x_i)^\top}{\sqrt{n}\|\x_{i}-\bar \x_i\|}=\sum_{i=1}^p \frac {\|\x_{i}-\bar \x_i\|}{\sqrt{n}}.
\end{align*}
Secondly,  we  show
$d_1\cas0$ and $d_2\cas0$.
By \eqref{levi2a}, \eqref{levi2b}, and \eqref{levi2c}, 
\begin{align*}
|d_1-\tilde d_1|&=\frac{1}{np}\left|\sum_{i=1}^{p}\sum_{j=1}^nx_{ij}^2-\sum_{i=1}^{p}\sum_{j=1}^n\left(x_{ij}-\bar x_i\right)^2\right|\\&=\frac{1}{np}\left|
\sum_{i=1}^{p}\sum_{j=1}^nx_{ij}^2-\sum_{i=1}^{p}\sum_{j=1}^nx_{ij}^2+n\sum_{i=1}^{p}\bar x_i^2\right|=\frac{1}{p}\sum_{i=1}^p\left|\bar x_i\right|^2.
\end{align*} 

Thus, by decomposition~\eqref{xij decomposition}, 
\begin{align*}
\frac{1}{p}\sum_{i=1}^p\left|\bar x_i\right|^2\le \left(\max_{1\le i\le p}|\bar x_i|\right)^2
&=\left(\max_{1\le i\le p}\left|\sum_{j=1}^n\frac{\sqrt{\rho}\eta_j+\sqrt{1-\rho}\xi_{ij}}{n}\right|\right)^2\\
&\le \left(\left|\sum_{j=1}^n\frac{\sqrt{\rho}\eta_j}{n}\right|+\max_{1\le i\le p}\left|\sum_{j=1}^n\frac{\sqrt{1-\rho}\xi_{ij}}{n}\right|\right)^2.\end{align*}
By Proposition~\ref{prop:bai}, almost surely, there exists $N_3(m,M)$ such that for all $n>N_3(m,M)$, we have $|\max_{1\le i\le Mn}|$ $n^{-1}\sum_{j=1}^n\xi_{ij}||\le m^{-1}$. Hence, by $n,p\to\infty$ with $p/n\to c>0$, for all $n\ge\max(N_1(m),N_3(m,(c+1/m)))$, we have
$-m^{-1}<\max_{1\le i\le (c-m^{-1})n}|n^{-1}\sum_{j=1}^n\xi_{ij}|\le\max_{1\le i\le p}|n^{-1}\sum_{j=1}^n\xi_{ij}|\le\max_{1\le i\le (c+m^{-1})n}$ $|n^{-1}\sum_{j=1}^n\xi_{ij}|<m^{-1}$. Therefore, as $n\to\infty$, $\max_{1\le i\le p}|n^{-1}\sum_{j=1}^n\xi_{ij}|\cas 0$.
As a result, by the triangle inequality and law of large numbers, 
\begin{align}
|d_1-\tilde d_1|\le\left(\left|\sum_{j=1}^n\frac{\sqrt{\rho}\eta_j}{n}\right|+\max_{1\le i\le p}\left|\sum_{j=1}^n\frac{\sqrt{1-\rho}\xi_{ij}}{n}\right|\right)^2\cas 0.
\label{levi2}
\end{align}
Thus, $d_1\cas 0$ because $\tilde d_1=(np)^{-1}\sum_{i=1}^{p}\sum_{j=1}^nx_{ij}^2-1\cas 0$. 

In contrast,
\begin{align*}
|d_2-\tilde d_2|&\le\frac{1}{p\sqrt{n}}\sum_{i=1}^p\left|\sqrt{\sum_{j=1}^n(x_{ij}-\bar x_i)^2}-\sqrt{\sum_{j=1}^nx_{ij}^2}\right|.\end{align*}
Note that $|\sqrt{r}-\sqrt{s}|\le |r-s|/\sqrt{s}$ $(r,s\ge0)$. Let $r=\frac{1}{n}\sum_{j=1}^n (x_{ij}-\bar{x_i})^2$ and $s= \frac{1}{n}\sum_{j=1}^n x_{ij}^2$. Then,
$|r-s|=(\bar x_i)^2$. 
By \eqref{j average}, $s\ge$\eqref{unc}. Thus, 
$$|d_2-\tilde d_2|\le
\frac{1}{p}\sum_{i=1}^p|\bar x_i|^2 \cdot \left(\frac{1}{n}\sum_{j=1}^nx_{ij}^2\right)^{-1/2}.$$ 
The first factor $\frac{1}{p}\sum_{i=1}^p\left|\bar x_i\right|^2$ on the right side is $|d_1-\tilde d_1|$, which converges almost surely to 0 by \eqref{levi2}. By \eqref{suffices}, almost surely, $\liminf_{\substack{n,p\to\infty\\ p/n\to c}}\left(\frac{1}{n}\sum_{j=1}^nx_{ij}^2\right)^{-1/2}\le 1$.
Thus, $|d_2-\tilde d_2|\cas0$.
Therefore, $d_2\cas 0$ since $\tilde d_2\cas 0$.
\end{proof}
 Proposition~\ref{prop:ji3}~(\cite[Theorem~1.2]{Ji}) assumes that $x_{ij}$ $(1\le i \le p;$ $1\le j\le n)$ are i.\@i.\@d.\@ random variables with finite second moments.
To prove the proposition, Jiang directly
applied Proposition~\ref{prop:bai} to establish \eqref{suffices2}.
In contrast,  Theorem~\ref{thm:equilsd2} assumes $\Exp(x_{ij}x_{kj})=\rho$, so
\eqref{suffices2} is not immediate from Proposition~\ref{prop:bai}.
Therefore, we take advantage of the decomposition~\eqref{xij decomposition} to represent $x_{ij}$ as a linear combination of independent, standard normal random variables $\eta_{j}$ and $\xi_{ij}$ for $1\le i\le p,\, 1\le j\le n$ to prove~\eqref{suffices2}.

\section{Analysis of stopping rules for principal components and factors}
\label{sec:components retention rules}
In the following subsections, we mathematically formulate the Guttman-Kaiser criterion and CPV rule. Under an ENP with $0\le \rho<1$, the limiting proportion of principal components and factors retained by the Guttman-Kaiser criterion and CPV rule are then examined using the LSDs of sample correlation matrices and sample covariance matrices. Moreover, we compare these findings to elaborate on Kaiser's observation that the Guttman-Kaiser criterion underfactors for small $c$.

\subsection{Guttman-Kaiser criterion}\label{subsec:GK}
Guttman-Kaiser criterion is a popular tool for assessing the dimensionality of empirical data.
This criterion is based on the arithmetical mean of the eigenvalues of a sample covariance matrix~\cite[p.~47]{jackson}. The principal component with an eigenvalue lower than the average will be deleted. 
Because the average eigenvalues of a sample correlation matrix are one, any principal component with the corresponding eigenvalue less than one is discarded. 
Guttman-Kaiser criterion asserts that each retained principal component should explain more variation than a single variable\textemdash which will be 1.0 when all variables are standardized.
For EFA, Guttman-Kaiser criterion keeps factors corresponding to eigenvalues larger than 1.0 from the sample correlation matrix $\bR$.
\comment{For EFA, H.~F.~Kaiser who introduced Guttman-Kaiser criterion adverted to the following~\cite{kai}: 
\begin{Quotation}\label{quotation}\rm ...
Humphreys (personal communication, 1984) asserts that, when the number $p$ of attributes is large and the ``average'' intercorrelation is small, the Kaiser-Guttman rule will overfactor. Tucker (personal communication, 1984) asserts that, when the number of attributes $p$ is small and the structure of the attributes is particularly clear, the Kaiser-Guttman rule will underfactor. ...\end{Quotation}
Here, ``overfactor''~(``underfactor'', resp.) means ``overestimate''~(``underestimate'', resp.) the number of factors in the factor model. 
According to~\cite{kai2}, `the ``average'' intercorrelation' is meant by $\rho$ and ``the structure of the attributes is particularly clear''  corresponds to an equi-correlated structure.
We regard ``the number of attributes $p$ is small'' as $c\to 0.$ \cite{Criterion} presumed that in zero-factor model for EFA, half the eigenvalues of sample correlation matrices would be greater than unity with some variation caused by sampling error. 
By simulation study, \cite{Criterion} observed that the number of factors Guttman-Kaiser criterion retains is about $p/2$ for independent ($\rho=0$) normal random variables (zero communality = zero factors). 
In this subsection, we precisely compute $q/p$ in $n,p\to\infty$, $p/n\to c>0$ where $q$ is the number of principal components or factors that the Guttman-Kaiser criterion preserves.
We apply Mar\v{c}enko-Pastur distributions to illustrate the limiting behavior of the Guttman-Kaiser criterion. By this result, we expound Kaiser's observation in Quotation~\ref{quotation} and a simulation study of \cite{Criterion} about how many factors Guttman-Kaiser criterion retains in EFA.}

Let $\M$ be a real symmetric positive semi-definite matrix of order $p$. Suppose that Guttman-Kaiser criterion retains $q$ eigenvalues of $\M$. Then, the ratio $q/p$ is represented as follows. Here, for a distribution function $F$, $\overline{F}$ represents $1-F$, which is called the \emph{complementary distribution} function of $F$.
\begin{definition}\label{def:gk}
For a real symmetric matrix $\M$ of order $p$, define the following 
$$GK^\M=\overline{F^\M}\left(\frac{1}{p}\Tr \M\right).$$
\end{definition}
For $\M$, we  consider a sample correlation matrix $\bR$ and a sample covariance matrix $\S$. Here, $\bR$ and $\S$ are formed from an ENP with $0\le \rho<1$. We  calculate the limits of random variables $GK^{\bR}$ and $GK^\S$ in $n,p\to \infty$ with $p/n\to c>0$.
\begin{definition}
For $c>0$ and $0\le \rho<1$, define the following random variable
$$GK_{c,\rho}=\overline{F_{c,1}}\left(\frac{1}{1-\rho}\right).$$
\end{definition}

\begin{thm}\label{conv3} Suppose $X_1,\ldots,X_n\stackrel{\mbox{\rm i.\@i.\@d.\@}}{\sim} \mathrm{N}_p(\bm{\mu},\,\mathbf{D}\Rc(\rho)\mathbf{D})$ for a deterministic vector $\bm{\mu}\in\R^p$, a deterministic nonsingular diagonal matrix $\mathbf{D}\in\R^{p\times p}$, and $0\le \rho< 1$. Suppose $n,p\to \infty$ with $p/n\to c>0$. Then almost surely, 
\begin{enumerate}\rm
    \item $GK^\bR\to GK_{c,\rho},$
    \item $\,GK^\S\to GK_{c,\rho}$, for $\bm{\mu}=\bm{0}$ and $\mathbf{D} = \sigma\mathbf{I}$ with $\sigma > 0$.
\end{enumerate}
\end{thm} 
\begin{proof}
(1) Because $\Tr \bR/p=1$ is a continuity point of $F_{c,1-\rho}$,  Theorem~\ref{thm:equilsd2} implies that, in limit $n,p\to\infty,p/n\to c>0$, almost surely,  $GK^{\bR}=\overline{F^{\bR}}(1)\to \overline{F_{c,1-\rho}}(1)=GK_{c,\rho}$.
\medskip

(2) Because $\Tr \S=\sum_{i=1}^{p}\sum_{j=1}^{n}x_{ij}^2/n=0$ implies $x_{ij}=0$ and because $x_{ij}$ are continuous random variables, $\Prb(\Tr \S/p=0)=0$. As a result, $\Tr \S/p$  is a continuity point of the distribution function $F_{c,\sigma^2(1-\rho)}$ almost surely.
Because $GK^\S= \overline{F^\S}\left( \Tr \S/p\right)$ and 
Theorem \ref{thm:equilsd},  in the limit of $n,p\to\infty$ with
 $p/n\to c>0,$ the following random variable
\begin{align}
Z_{p,\rho}:= GK^\S  - \overline{ F_{c,\sigma^2(1-\rho)}}\left(\frac{1}{p}\Tr \S\right) \label{zprho}
\end{align}
converges to 0 almost surely.
By Lemma~\ref{lem:lim2},  
${\Tr \S}/{p}\cas \sigma^2$.
Besides,
 $\overline{F_{c,\sigma^2(1-\rho)}} $ is continuous at $\sigma^2$ because $\sigma^2>0$.
Thus, $\overline{F_{c,\sigma^2(1-\rho)}}(\Tr \S /p )\cas \overline{F_{c,\sigma^2(1-\rho)}}(\sigma^2)$.
Hence, in the limit of $n,p\to\infty$ with
 $p/n\to c>0,$ almost surely, $$GK^\S = Z_{p,\rho}+\overline{F_{c,\sigma^2(1-\rho)}}(\Tr \S/p
 )\to \overline{F_{c,\sigma^2(1-\rho)}}(\sigma^2)=\overline{F_{c,1-\rho}}(1)$$
 which is $GK_{c,\rho}.$
\end{proof}

\begin{thm}\label{conv3b} 
Suppose that $x_{ij}$ $(1\le i \le p;$ $1\le j\le n)$ are \emph{i.\@i.\@d.\@} centered random variables with variance $\sigma^2$ $(0<\sigma^2<\infty)$. Suppose $n,p\to \infty$, and $p/n\to c>0$. Then, the following assertions hold:
\begin{enumerate}\rm
\item \label{maincor} Almost surely, $
GK^{\bR} \to GK_{c,0}.$
\item \label{main} If $c\le 1$, then almost surely, $GK^\S \to GK_{c,0}$.
\item \label{main2} If $\Exp(x_{ij}^{12})<\infty$, then, in probability, $GK^\S \to GK_{c,0}$. 
\end{enumerate}
\end{thm}
\begin{proof} (1) The proof is the proof of Theorem~\ref{conv3} except that $\rho$ should be 0, and Theorem \ref{thm:equilsd2} should be Proposition~\ref{prop:ji3}.

(2) The proof is the proof of Theorem~\ref{conv3} except that $\rho$ should be 0, and Theorem \ref{thm:equilsd} should be Proposition~\ref{prop:MP LT}.
The weak convergence to $F_{c,\sigma^2}$ becomes the pointwise convergence because $F_{c,\sigma^2}$ is continuous by $0<c\le 1$.

(3) By $GK^\S=\overline{F^\S}\left(\frac{1}{p}\Tr \S\right)$,
$\Exp   \left| GK^\S - \overline{F_{c,\sigma^2}} \left(\frac{1}{p}\Tr \S\right ) \right|$ is at most $\Exp\sup_{x\in\R}$   $\left|F^\S(x)-F_{c,\sigma^2}(x)\right|.
$
Hence, by the following Proposition~\ref{prop:gotze}, and Proposition~\ref{prop:gotze2},  in the limit of $n$ and $p$, $Z_{p,0}$ of \eqref{zprho}
converges to 0 in mean, and thus in probability.
By the law of large numbers, almost surely, 
$\Tr \S/{p}$ converges to $\sigma^2$, and thus in probability. Since $\sigma^2>0$, $\overline{F_{c,\sigma^2}}$ is continuous at $\sigma^2$. Hence, by \cite[p.~7, Theorem~2.3]{Van98},
$\overline{F_{c,\sigma^2}}(\Tr \S /p )$ converges in probability to $ \overline{F_{c,\sigma^2}}(\sigma^2)$.
As a result, in the limit $n,p\to\infty$ with $p/n\to c>0$, it holds in probability that $GK^\S  = Z_{p,0}+\overline{F_{c,1}}(\Tr \S /p
 )\to \overline{F_{c,\sigma^2}}(\sigma^2)=\overline{F_{c,1}}(1)$.
\end{proof}

\begin{prop}[\protect{\cite[{Theorem 1.2}]{GT}}]\label{prop:gotze}
Assume the following conditions.
\begin{enumerate}[(i)]\rm
\item \label{cond:independent}
$x_{ij}$ $(1\le i\le p,\ 1\le j\le n)$ are independent, centered random variables.
\item \label{cond:twelfth moment}
 $M_{12}:=\max_{1\leq i\leq p,\;1\leq j\leq n}\Exp \left |x_{ij}\right |^{12}<\infty.$
\item \label{cond:unit variance}
$\Var(x_{ij})=1$ $(1\le i\le p,\ 1\le j\le n)$ and $p/n=c$.
\end{enumerate}
Then, if
$0<\Theta_1\leq c\leq \Theta_2<\infty$ and $|c-1|\geq\theta>0$ for some constants $\Theta_1,\Theta_2,\theta$,
there is an absolute constant $C(\theta,\Theta_1,\Theta_2)$ such that
$$\Exp\left(\K (F^\S,F_{c,1}) \right) \leq C(\theta,\Theta_1,\Theta_2)M_{12}^{1/6}n^{-1/2}.$$
\end{prop}

\begin{prop}[\protect{\cite[{Theorem 1.2}]{GT2}}]\label{prop:gotze2}Assume the premises \eqref{cond:independent}, \eqref{cond:twelfth moment}, and \eqref{cond:unit variance} of Proposition~\ref{prop:gotze}.
Then,
if $1\geqq c > \theta > 0$ for some constant $\theta$,
 there is a positive constant $C(\theta)$ such that
$$\Exp\left(\K(F^\S,F_{c,1})\right)\leq C(\theta)M_{12}^{1/6}n^{-1/2}.$$
\end{prop}

Figure~\ref{Fig:qgk1} is the graphs of $GK_{c,\rho}$ over $c\in(0,\, 20)$, for $\rho={0,0.3,0.5,0.8}$. 
This Figure elucidates Quotation~\ref{quotation} and \cite{Criterion}.
 \begin{figure}[ht]\centering
\includegraphics[scale=.56]{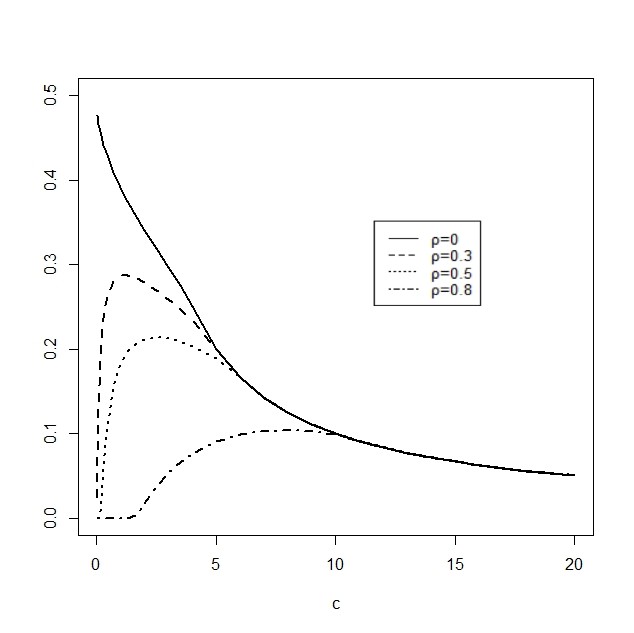}
\caption{$GK_{c,\rho}$}
\label{Fig:qgk1}
\end{figure} 

The second sentence of Quotation~\ref{quotation} is conformed to the following Theorem~\ref{thm:fig}~\eqref{assert:zero}.
In view of Theorem~\ref{conv3} and Theorem~\ref{conv3b}, one may think of $GK_{c,\rho}\le1/2$ for all $c>0$ and all $\rho\in [0,1)$.
This indeed follows from Theorem~\ref{thm:fig}~\eqref{assert:phase transition} and Theorem~\ref{thm:fig}~\eqref{assert:nonincreasing in rho}.

If  $n, p$ are large and $p/n$ is sufficiently large, then Guttman-Kaiser criterion may retain all the $\min(n,p)$ principal components or factors corresponding to positive eigenvalues, from the sample covariance matrices and the sample correlation matrices. 
Here, $p$ is the dimension of the population, and $n$ is the sample size.
See Figure~\ref{fig:qjbs} for $pGK^{\bR}/{\min(n,p)}$.

\begin{thm}\label{thm:fig}~
\begin{enumerate}\rm
 \item \label{assert:zero} $\lim_{c\downarrow 0} GK_{c,\rho} =0\,(0<\rho<1).$
\item \label{assert:phase transition} $\lim_{c\downarrow 0} GK_{c,0} =1/2\, (\rho=0)$ and $GK_{c,0}$ is strictly decreasing in $c>0$.
\item \label{assert:nonincreasing in rho}For any $c>0$, $GK_{c,\rho}$ is nonincreasing in $\rho\in[0,\, 1)$.
\item \label{assert:all nonzero} If $0\le \rho<1$, then $c\ge(1/\sqrt{1-\rho}+1)^2\iff GK_{c,\rho}=1/c$.  
      \end{enumerate}
\end{thm} 

\begin{proof} \eqref{assert:zero} $GK_{c,\rho}=\overline{F_{c,1}}(1/(1-\rho))=0$ if and only if $b_{1}(c)<1/(1-\rho)$.
By $\lim_{c\downarrow0} b_{1}(c)=1<1/(1-\rho)$, $\lim_{c\downarrow0} GK_{c,\rho}=\lim_{c\downarrow0} \overline{F_{c,1}}(1/(1-\rho))=0$.

\medskip
\noindent\eqref{assert:nonincreasing in rho} Because
$GK_{c,\rho}=\overline{F_{c,1}}(1/(1-\rho))$, the complementary distribution function $\overline{F_{c,1}}(x)$ is nonincreasing in $x$, and $1/(1-\rho)$ is strictly increasing in $\rho\in [0,1)$.

\medskip\noindent
\eqref{assert:all nonzero}  Assume
$GK_{c,\rho}=\overline{F_{c,1}}((1-\rho)^{-1})=1/c$. Because the mean of Mar\v{c}enko-Pastur distribution of index $c$ and scale parameter $a>0$ is $a$, $1>\overline{F_{c,1}}((1-\rho)^{-1})=GK_{c,\rho}$. Therefore $c>1$. Hence, Mar\v{c}enko-Pastur distribution of index $c$ has the probability mass $1-1/c$ at 0. Therefore, $(1-\rho)^{-1}\le a(c)$. The converse is easy. Now, we  solve  $(1-\rho)^{-1}\le a(c)$ and $c>1.$ By taking the square roots of both hand sides of $(1-\rho)^{-1}\le a(c)$, $(1-\rho)^{-1/2}\le 1-\sqrt{c}$ or $-(1-\rho)^{-1/2}\ge 1-\sqrt{c}$. The former is impossible, because $c>1$. Therefore, we have a solution $(1+(1-\rho)^{-1/2})^2\le c$.

\medskip\noindent
\eqref{assert:phase transition}
If a random variable $X_c$ follows the Mar\v{c}enko-Pastur distribution $F_{c,1}$, then as $c\to 0$, a random variable $X_c'=(2\sqrt{c})^{-1}(X_c-1)$ converges in distribution to a random variable $X'$ 
that follows \emph{Wigner's semi-circle law} with density function $2\pi^{-1}\sqrt{1-x^2}$ for $|x|\le1$ and 0 otherwise. By this, $
1-\lim_{c\downarrow 0}GK_{c,0}=\lim_{c\downarrow 0}F_{c,1}(1)= \lim_{c\downarrow 0}\Prb(X_c\le 1)=\lim_{c\downarrow 0}\Prb(X_c'\le 0)=\Prb(X'\le 0)={1}/{2}.$
Thus, $\lim_{c\downarrow0}GK_{c,0}=1/2$.
 
Next, we  prove that $GK_{c,0}$ is strictly decreasing in $c>0$.
By \eqref{assert:all nonzero} of this theorem, for $c\ge4$, $GK_{c,0}=1/c$ is strictly decreasing in $c$.
One can easily check that the cumulative distribution function of Mar\v{c}enko-Pastur distribution $F_{c,1}$ is equal to 
\begin{align*}
    \begin{cases}
    \frac{c-1}{c}\I_{x\in[0,a_1(c))}+\left(\frac{c-1}{2c}+F(x)\right)\I_{x\in[a_1(c),b_1(c)]}+\I_{x\in[b_1(c),\infty)}, &(c>1)\\
    F(x)\I_{x\in[a_1(c),b_1(c)]}+\I_{x\in[b_1(c),\infty)},&(0<c\le 1),
    \end{cases}
\end{align*}
where
\begin{multline*}
F(x)=\frac{1}{2\pi c}\left({\pi c+\sqrt{(b_1(c)-x)(x-a_1(c))}-(1+c)\arctan\frac{r(x)^2-1}{2r(x)}}\right)\\+\frac{1}{2\pi c}\left({(1-c)\arctan\frac{a_1(c)r(x)^2-b_1(c)}{2(1-c)r(x)}}\right) 
\end{multline*}
and
$r(x)=\sqrt{(b_1(c)-x)/(x-a_1(c))}$.
By calculation, $-2\pi c^2 d\overline{F_{c,1}}(1)/dc$ is  $f(c)$ for $0<c\le1$ and $\pi+f(c)$ for $1<c<4$,
by letting $f(c)$ be $\arctan \left( \sqrt {c}/\sqrt {4-c}\right)+\arctan \left( {{\sqrt {c} \left( c-3 \right) }/({\sqrt {4-c} \left( -1+c\right) }}\right) -\sqrt {c}\sqrt {4-c}$. 
Because $f(0)=0$ and $f'(c)=\sqrt{4-c}/\sqrt{c}>0$ for $0< c<4$, we have $f(c)>0$ for $0<c<4 $.
Hence, $dGK_{c,0}/dc<0$ for $0<c<4$. 
Consequently, $\overline{F_{c,1}}(1)$ is strictly decreasing in $0<c<4$.
\end{proof}

Figure~\ref{fig:qjbs} consists of the graphs of $pGK^{\bR}/\min(n,p)$ under assumption Theorem~\ref{conv3} with $n=1000$ and $p=50,70,\ldots,19980,20000$; and $\rho=0\,\text{(solid)},\,0.3$ $\;\text{(dashed)},\,0.5\,\text{(dotted)},\,0.8\,\text{(dash-dot)}$. We can find $pGK^{\bR}/\min(n,p)$ is equal to 1 if $c>({1/\sqrt{1-\rho}}+1)^2$ for large $n$ and $p$.

\begin{figure}[ht]\centering
\includegraphics[scale=.55]{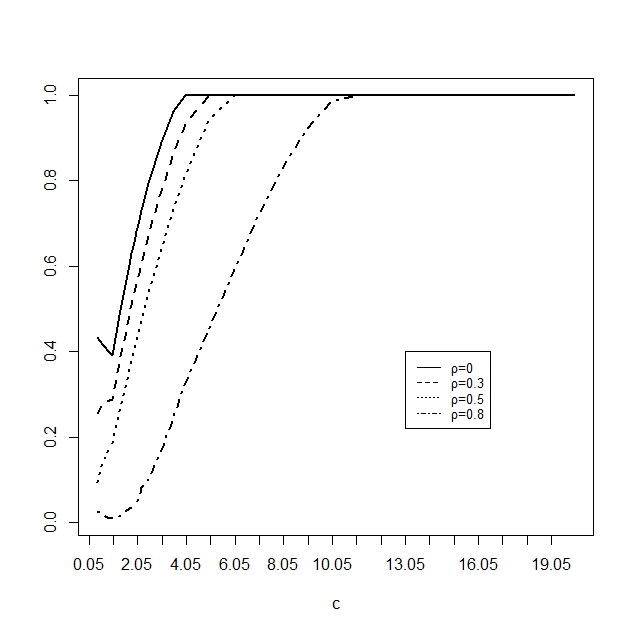}
\caption{$pGK^{\bR}/{\min(n,p)}$}
\label{fig:qjbs} 
\end{figure}

To expound the first sentence of Quotation~\ref{quotation},
 we compare Guttman-Kaiser criterion to another stopping rule of the number of principal components and factors, in the following subsections.
\subsection{CPV rule (Cumulative-percentage-of-total-variation rule)}\label{appendixB}
In PCA and EFA, each eigenvalue of the sample correlation matrices and the sample covariance matrices represents the level of variation explained by the associated principal components and factors. A simple and popular stopping rule has been related to the proportion of the sum of all eigenvalues of the sample correlation matrices or the sample covariance matrices explained by $q$ number of principal components and factors retained, say $t\in(0,\,1)$.

Let $\M$ be a real symmetric matrix semi-definite matrix of order $p$.
Let all the $p$ eigenvalues of $\M$ be 
\begin{align}\label{eigenvalues1}
\lambda_1\ge\lambda_2\ge\cdots\ge\lambda_p\ge 0.
\end{align}
For $t\in(0,\,1)$, the number $q$ of eigenvalues \emph{CPV rule} retains is, by definition,
the maximum nonnegative integer $q<p$ such that  $\sum_{i=1}^q\lambda_i\big/\sum_{i=1}^p\lambda_i\le t$ and $\lambda_q>\lambda_{q+1}$. 
For convenience, we set $\lambda_0=\lambda_1+1.$
Thus, there is indeed the maximum $q$, because $\sum_{i=1}^0\lambda_i/\sum_{i=1}^p\lambda_i\le t$ and $\lambda_0>\lambda_1$.
Since we required $\lambda_q>\lambda_{q+1}$, 
once CPV rule takes an eigenvalue $\lambda$, CPV rule takes all eigenvalues $\lambda_i$ such that $\lambda_i=\lambda$.

\begin{definition}\label{def:CPM}
For a threshold $t\,(0<t<1)$, define the following 
$$CP^{\M}(t)=\max\Set{\frac{q}{p} | \sum_{i=1}^q \lambda_i\big/\sum_{i=1}^p \lambda_i \le t,\  \lambda_q>\lambda_{q+1}\ \mbox{and}\ 0\le q<p}$$
where we set $\lambda_{0}=\lambda_1+1.$
\end{definition}
$CP^{\M}(t)\ne 1$, because $t<1$ and $\M$ has at least one positive eigenvalue.
Since the eigenvalues of matrix $\M$ follows~\eqref{eigenvalues1}, the eigenvalues of matrix $k\M$ are~\eqref{eigenvalues1} multiplied by $k$ for $k>0$. Thus, $CP^{k\M}=CP^{\M}$ for any $k>0$. 

We  study the limiting proportion of principal components and factors retained by CPV rule with threshold $t\in(0,\,1)$. 
\begin{definition}
\label{def:Gc}
For any $c,\sigma^2>0$, define a function
\begin{align*}
\G{c}{\sigma^2}(x)&=\frac{\int_{(-\infty,x]}\lambda d\MPmeas{c}{\sigma^2}(\lambda)}{\int_{\R}\lambda d\MPmeas{c}{\sigma^2}(\lambda)}\qquad(x\in\R).
\end{align*}
\end{definition}
A \emph{defective distribution function} is, by definition, a right-continuous, nondecreasing function on $\R$ that vanishes at $-\infty$. 
By the convention that $\inf\varnothing=\infty$ and that a set without a lower bound has an infimum $-\infty$, we define the \emph{generalized inverse} \cite{EH} of a possibly defective distribution function $F$ as
\begin{align*}
F^{-}(t)=\inf\Set{x\in\R | F(x)\ge t},\qquad (t\in\R).
\end{align*}
A \emph{quantile function} is the generalized inverse of a distribution function~\cite[p.~304]{Van98}.

Let us consider the complementary distribution function $\overline{F_{c,1}}=1-F_{c,1}$ applied to the quantile function $(G_{c,1})^-$ of  $(1-t)/(1-\rho).$ Hereafter, for a monotone function $f:\R\to\R$, we set $f(-\infty)=\lim_{x\downarrow-\infty}f(x)$ and $f(\infty)=\lim_{x\uparrow \infty}f(x)$.

\begin{definition}\label{def:cp}
For $c>0$, $0\le \rho<1$, and $0<t<1$, define a function
$$CP_{c,\rho}(t)=\overline{F_{c,1}}\left(\qtl{\G{c}{1}}\left(\frac{1-t}{1-\rho}\right)\right).$$
\end{definition}
The main theorems of this subsection are the following:
\begin{thm}\label{thm:EqCor CPMconvCPc} 
Suppose $X_1,\ldots,X_n\stackrel{\mbox{\rm i.\@i.\@d.\@}}{\sim} \mathrm{N}_p(\bm{\mu},\,\mathbf{D}\Rc(\rho)\mathbf{D})$ for a deterministic vector $\bm{\mu}\in\R^p$, a deterministic nonsingular diagonal matrix $\mathbf{D}\in\R^{p\times p}$, and $0\le \rho< 1$. Suppose $n,p\to \infty$ with $p/n\to c>0$. Then almost surely, for any $t\in[\rho,\, 1)$,
\begin{enumerate}\rm
    \item $CP^\bR(t)\to CP_{c,\rho}(t),$
    \item $\,CP^\S (t) \to CP_{c,\rho}(t)$, for $\bm{\mu}=\bm{0}$ and $\mathbf{D} = \sigma\mathbf{I}$ with $\sigma > 0$.
\end{enumerate}
\end{thm}
\begin{thm}\label{thm:null CPMconvCPc} 
Suppose that $x_{ij}$ $(1\le i \le p;$ $1\le j\le n)$ are \emph{i.\@i.\@d.\@} random variables with variance $\sigma^2$ $(0<\sigma^2<\infty)$. Suppose  $n,p\to \infty$ with $p/n\to c>0$. Then almost surely, for any $t\in (0,1)$,
\begin{enumerate}\rm
    \item $CP^\bR(t)\to CP_{c,0}(t),$
    \item $\,CP^\S (t) \to CP_{c,0}(t)$, for $\Exp{x_{ij}}={0}$.
\end{enumerate}
\end{thm}
\emph{\textbf{Proof sketch of Theorem~\ref{thm:EqCor CPMconvCPc}.}} 
Let $\M$ be $\bR$ or $\S$.
\begin{enumerate}
    \item 
We define a random distribution function
\begin{align}\label{def:GT}
G^\M(x)={\sum_{\lambda_i\le x} \lambda_i}\big/{\sum_{i=1}^p\lambda_i},\quad (x\in\R).
\end{align}
\item
We prove that  $\qtl{G^\M}(1-t)$ is the threshold of eigenvalues of $\M$ that CPV rule retains. In other words, $CP^{\M}(t)=\overline{F^\M}\left(\qtl{G^\M}(1-t)\right).$
\item \label{proc:2}
By Theorem~\ref{thm:equilsd2} (Theorem~\ref{thm:equilsd}, resp.), we prove that almost surely, $G^\bR$ ($G^\S$, resp.) converges pointwise to a defective distribution function $(1-\rho)G_{c,1-\rho}$ ($(1-\rho)G_{c,\sigma^2(1-\rho)}$, resp.). 
\item
By this, we derive that almost surely,  for any $u\in(0,\,1-\rho]$,  $ \qtl{G^\bR}(u)$ tends  to $\qtl{(1-\rho)\G{c}{1-\rho}}(u)$ and $ \qtl{G^\S }(u)$ does to $\qtl{(1-\rho)\G{c}{\sigma^2(1-\rho)}}(u)$.

\item
Then, we deduce that almost surely, for any $t\in[\rho,\,1)$, $$CP^\M (t)=\overline{F^\M}(\qtl{G^\M}(1-t))\to CP_{c,\rho}(t)=\overline{F_{c,1-\rho}}\left(\qtl{\G{c}{1-\rho}}\left(\frac{1-t}{1-\rho}\right)\right).$$
\end{enumerate}
For the proofs of Theorem~\ref{thm:EqCor CPMconvCPc} and Theorem~\ref{thm:null CPMconvCPc}, see Appendix.

\medskip
We contrast Theorem~\ref{thm:fig} for Guttman-Kaiser criterion, against
the following theorem for CPV rule:

\begin{thm}~
\begin{enumerate}
    
\item \label{assert:cont decreasing} $\lim_{c\downarrow 0}CP_{c,\rho}(t)= 1-\frac{1-t}{1-\rho}$ for $0\le \rho<t<1$. 

\item \label{assert:cp nonincreasing in rho} $CP_{c,\rho}(t)$ is nonincreasing in $\rho\in[0,\, 1)$ and nondecreasing in $t\in(0,\,1)$.

\item\label{assert:cpv 0} $CP_{c,\rho}(t)= 0$ if $0<t<\rho<1$.
\end{enumerate}
\end{thm}
\begin{proof} Let $s$ be $(1-t)/(1-\rho)$.
\eqref{assert:cont decreasing} By $0\le \rho<t<1$, $ (G_{c,1})^-\left(s\right)<\infty$.
Because of $F_{c,1}(a_1(c))=G_{c,1}(a_1(c))=0$ and Cauchy's mean value
	theorem, there is $x'$ such that $a_1(c)<x'< (G_{c,1})^-\left(s\right)$ and
	\begin{align*}
	 \frac{F_{c,1}\left((G_{c,1})^-\left(s\right)\right)}{G_{c,1}\left((G_{c,1})^-\left(s\right)\right)} = \frac{F'_{c,1}(x')}{G'_{c,1}(x')}=\frac{1}{x'}.
	\end{align*}
	This tends to $1$ in $c\downarrow0$, because
	$1\leftarrow a_1(c)<x'< (G_{c,1})^-\left(s\right)\le b_1(c)\to 1$. By
	$G_{c,1}((G_{c,1})^-(s))=s$, $\lim_{c\downarrow0}
	{F_{c,1}}(G_{c,1}^-(s))=s$ for $\rho<t$.

\medskip
\noindent\eqref{assert:cp nonincreasing in rho} It is because $CP_{c,\rho}(t)=\overline{F_{c,1}}$ $(\qtl{\G{c}{1}}(s))$, $G_{c,1}$ is nondecreasing, and
the complementary distribution function $\overline{F_{c,1}}$ is nonincreasing. 

\medskip
\noindent\eqref{assert:cpv 0} 
By $1>\rho>t>0$,  $s>1$. Because $G_{c,1}$ is a distribution function, $\qtl{\G{c}{1}}(s)=\inf\varnothing=\infty$. 
$\overline{F_{c,1}}$ is nonincreasing, so $CP_{c,\rho}(t)=\overline{F_{c,1}}(\infty)=0$.
\end{proof}
Because the centered sample covariance matrix $\tilde \S$ is invariant under shifting of variables, Figure~\ref{Fig:cpv} consists of the the graphs of $CP^{\tilde \S} (0.7)$ under the assumption of Theorem~\ref{thm:EqCor CPMconvCPc} with $n=1000$, $p=20,40,\ldots,19980,20000$; and $\rho=0\,\text{(solid)},\,0.3$ $\text{(dashed)},\,0.5\,\text{(dotted)},\,0.8\,\text{(dash-dot)}$. By Figure~\ref{Fig:cpv}, it seems that $CP_{c,\rho}(t)$ is decreasing in $c$.
\begin{figure}[ht]\centering
 		\includegraphics[scale=0.6]{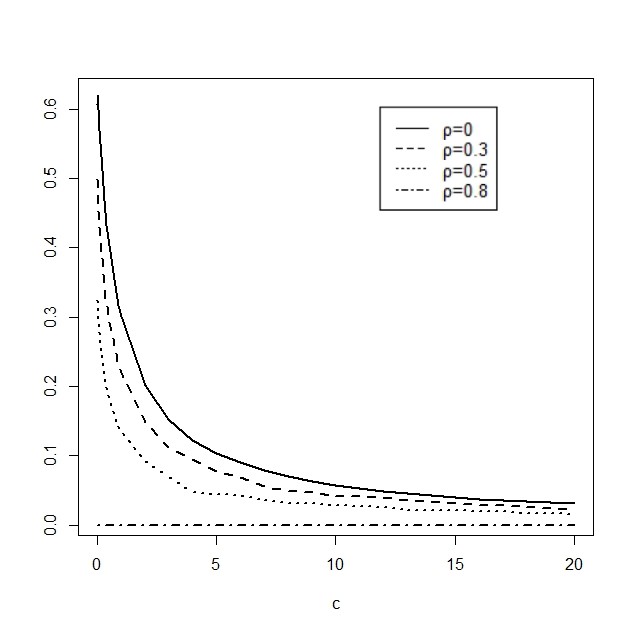}
\caption{$CP^{\tilde \S} (0.7)$ }
\label{Fig:cpv}
\end{figure} 

\begin{obs}\label{obs}
When the sample size $n$ and the dimension $p$ are large, we obtain:
\begin{enumerate}
\item The behavior of Guttman-Kaiser criterion (CPV rule, resp.) corresponds to $GK_{c,\rho}$ ($CP_{c,\rho}(t)$, resp.), by Theorem~\ref{conv3} (Theorem~\ref{thm:EqCor CPMconvCPc}, resp.).

\item \label{GK decreasing in rho} Both of $GK_{c,\rho}$ and $CP_{c,\rho}(t)$ are nonincreasing in the equi-correlation coefficient $\rho$.

\item\label{approach0.5} The limit of $GK_{c,\rho}$ in $c\downarrow0$ is $1/2$ for $\rho=0$ and 0 otherwise, but the limit of $CP_{c,\rho}(t)$ in $c\downarrow0$ is continuously nonincreasing in $\rho$.
\item From the combined Figure~\ref{fig:cpgk1} of Figure~\ref{Fig:qgk1} ($GK_{c,\rho}$) and Figure~\ref{Fig:cpv}  ($CP^{\tilde \S} (0.7)$), we could say that CPV rule retains sufficiently large number of principal components or factors in small $c$. In contrast, Guttman-Kaiser criterion retains small number of principal components or factors in small $c$. 
\end{enumerate}
\end{obs}
\begin{figure}[ht]\centering
\includegraphics[scale=.65]{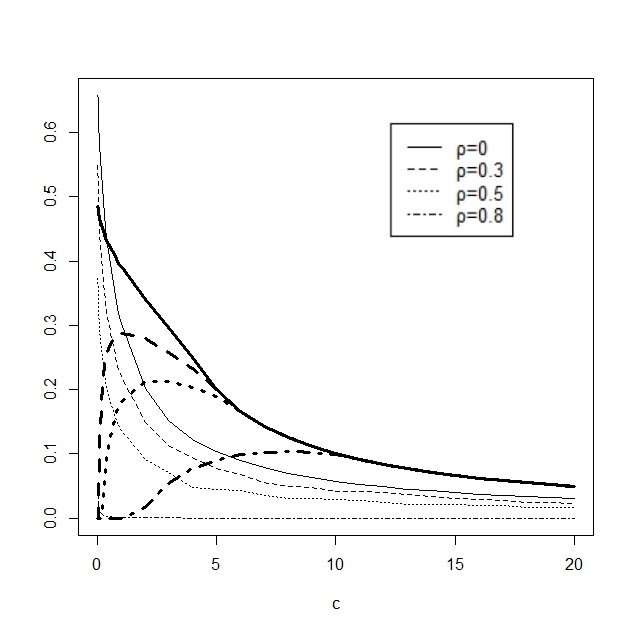}
\caption{Comparison of Guttman-Kaiser criterion (thick curves) and CPV rule.}
\label{fig:cpgk1} 
\end{figure}
\section{Empirical study of equi-correlation coefficient and Guttman-Kaiser criterion}\label{sec:app}

For the scaling parameter $1-\rho$ of Mar\v{c}enko-Pastur distribution of Theorem~\ref{thm:equilsd2},  the first author proposed $1-\lambda_1(\bR)/p$ in \cite{akama}:
\begin{thm}[\protect{Akama~\cite{akama}}]\label{thm:main}Let $\bR$ be a sample correlation matrix formed from a population $\N_p(\bm{\mu},\, \D \Rc(\rho)\D)$ for a deterministic vector $\bm{\mu}\in\R^p$, a deterministic nonsingular diagonal matrix $\D\in\R^{p\times p}$, and a deterministic constant $\rho\in[0,\,1)$. Suppose
 $p,n\to\infty$ and $p/n\to c\in(0,\,\infty)$. Then,
  \begin{align*}
  \frac{\lambda_1(\bR)}{p}\cas\rho.
\end{align*}
\end{thm}

With this, we expound Observation~\ref{obs} by real datasets such as a binary multiple sequence alignment (MSA) dataset~\cite{querde}, the microarray datasets~\cite{ramey}, the returns of S\&P500 stocks of specific periods~\cite{laloux}, and the households datasets~\cite{stat}. 

\subsection{Datasets from molecular biology}

For vaccine design, Quadeer et al.~\cite{querde}
considered a multiple sequence alignment (MSA) 
of a $p$-residue (site) protein with $n$ sequences where $p=475$ and $n=2815$. They converted
the MSA into a binary code following~\cite{dahi,halabi}, and then considered the correlation matrix $\bR$. From this, Quadeer et al.~\cite{querde,10.1371/journal.pcbi.1006409} detected signals 
by ingenious randomization for the MSA. At the same time, they considered an alternative method that employs Mar\v{c}enko-Pastur distribution. We also examine our study of Guttman-Kaiser criterion with Mar\v{c}enko-Pastur distribution, by using their binary MSA dataset of Quadeer et al.~\cite{querde,10.1371/journal.pcbi.1006409}. 

The binary MSA dataset~\cite{querde} is sparse by the heat map~(Figure~\ref{fig:medical}~(a)) of the dataset with hierarchical clustering on columns and rows.
As for the $p^2=225625$ entries of the correlation matrix $\bR$, the minimum, the first quantile, the median, the mean, the third quantile, and the maximum are as in Table~\ref{summary of PCM}.
\begin{table}[ht]\centering
\begin{tabular}{|c|c|c|c|c|c|}
\hline
Min.  &  1st Qu.  &   Median     &  Mean   & 3rd Qu.   &    Max. \\
\hline
$-0.2892210$ &$-0.0043780$ &$-0.0017447$  &$0.0068210$& $-0.0006157$ & 1.0000000 \\
\hline
\end{tabular}
\caption{The summary of the entries of the correlation matrix of the binary MSA dataset.}\label{summary of PCM}
\end{table}

Figure~\ref{fig:medical}~(b) is the heat map of
$\bR$ with hierarchical clustering on columns and rows.
\begin{figure}[ht]
    \centering
   \subfigure[]{\includegraphics[width=0.4\textwidth]{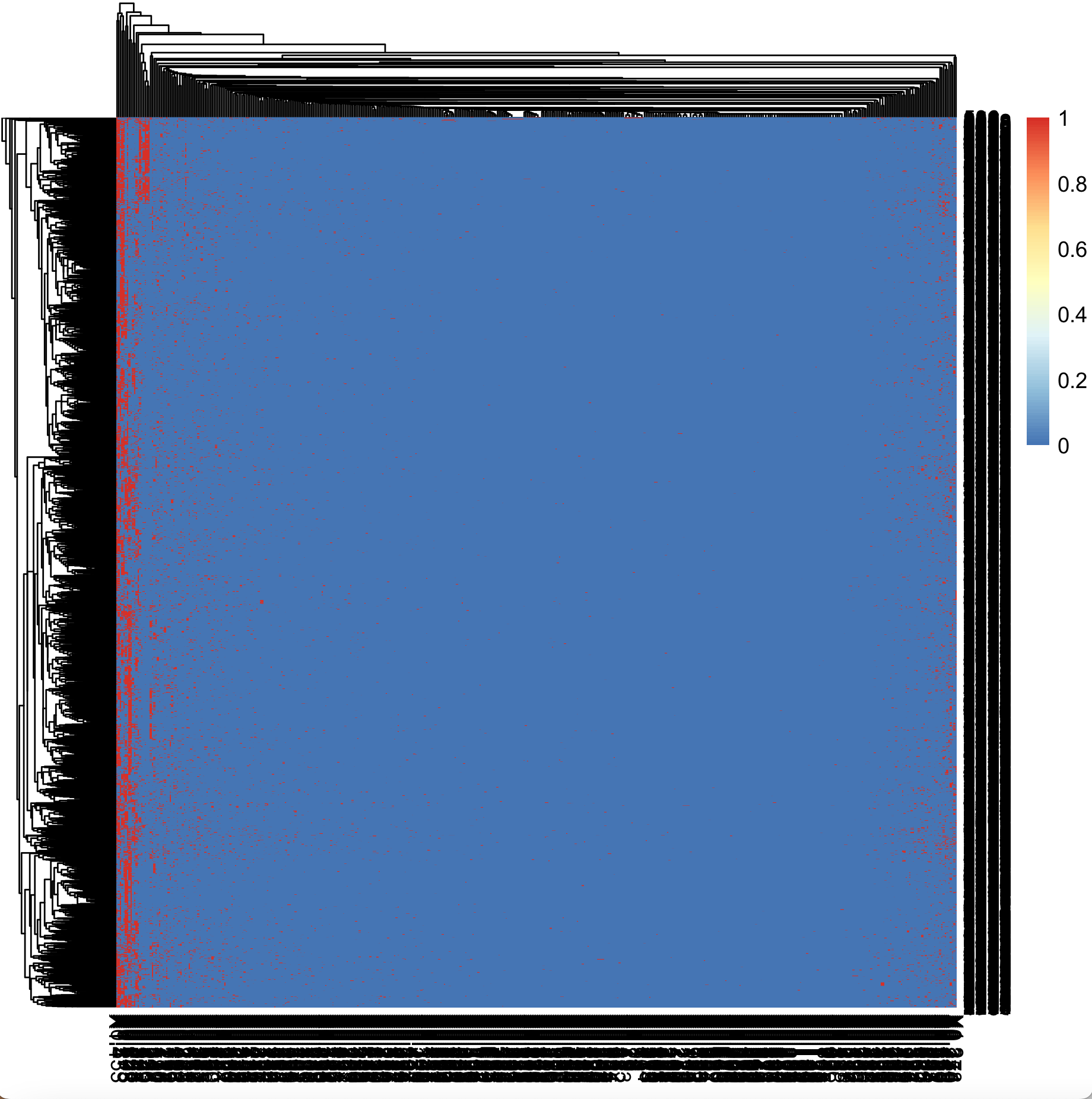}} 
    \subfigure[]{\includegraphics[width=0.4\textwidth]{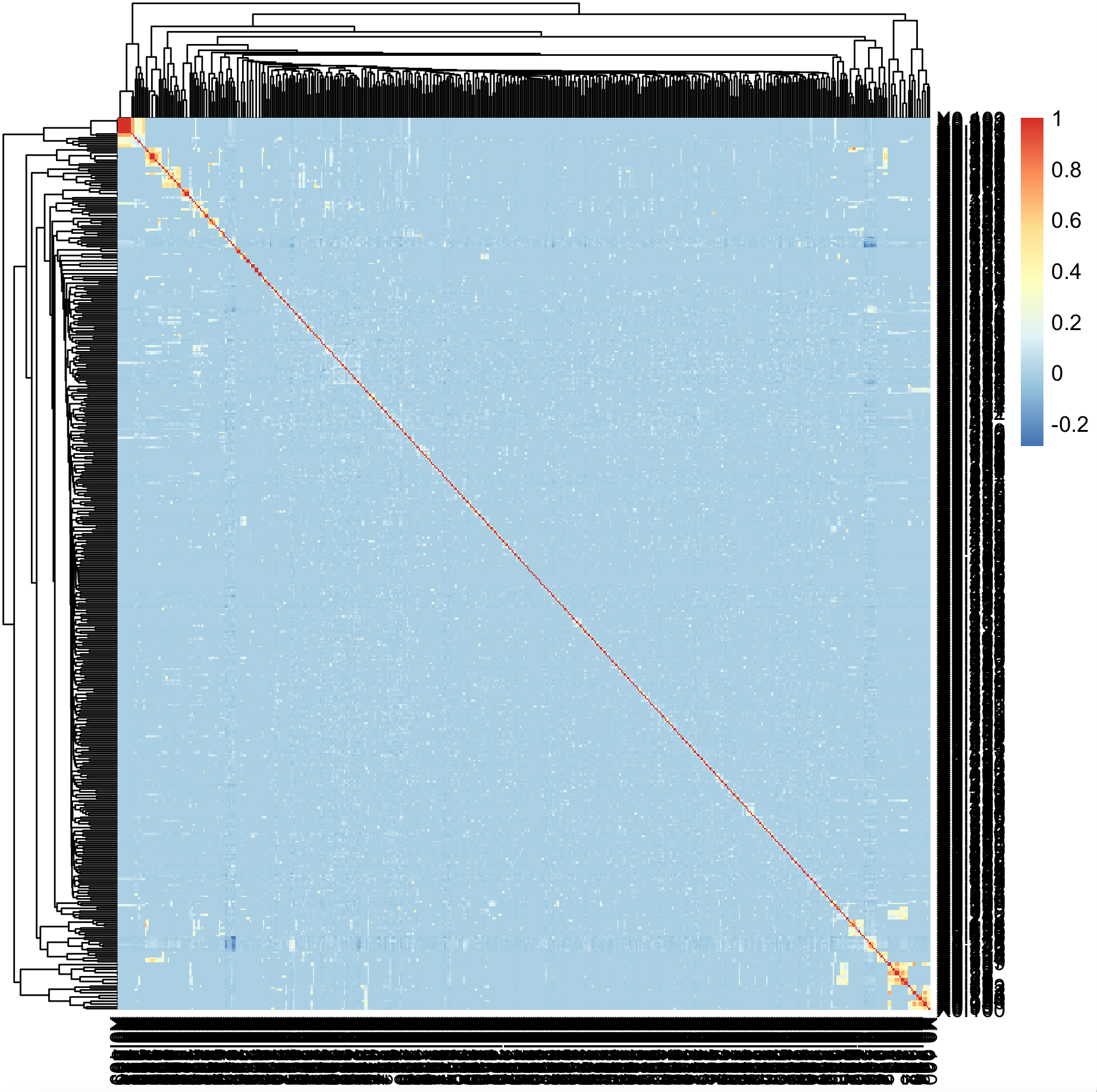}\label{heatmapperson}} 
    \subfigure[]{\includegraphics[width=0.67\textwidth]{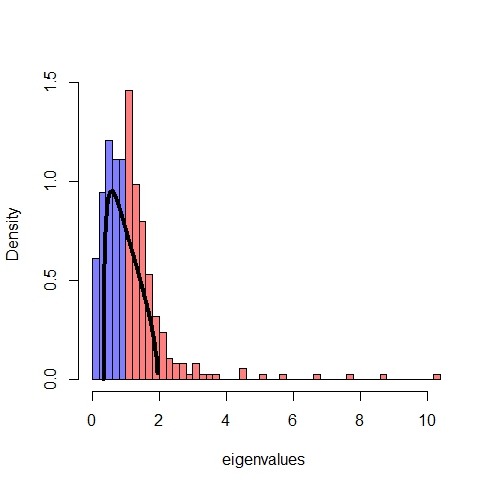}\label{pearson}} 
    \caption{\label{fig:medical}The binary MSA dataset. (a) The data matrix, (b) the correlation matrix $\bR$, and (c) the eigenvalues of $\bR$.}
\end{figure}
 Figure~\ref{fig:medical}~(c) is the histogram of the eigenvalues of $\bR$. The red bins are for the eigenvalues more than 1 and the blue bins are for the eigenvalues less than 1. The black solid curve is the density of Mar\v{c}enko-Pastur distribution with index $p/n=0.16$ and scale $1-\lambda_1(\bR)/p=0.98$.
There is a bin-width such that the histogram fits the density well.  
The $GK^\bR=0.39$ has 10\% error from the estimated value $GK_{{0.16},0.02}=0.44$. This may be related to Observation~\ref{obs}~(\ref{approach0.5}).

Next, we consider 16 microarray datasets from \cite{ramey}. Table~\ref{table:micro} is the list of the author of a microarray dataset, $p/n$, $\lambda_1(\bR)/p$, $GK^{\bR}$, $GK_{p/n,\lambda_1(\bR)/p}$, $CP^\bR(.7)$ and $p$, in the increasing order of $p/n$. Here $p$ is the number of features and $n$ is the number of observations.
\begin{longtable}{ |c|c|c|c|c|c|c|c| } 
			\hline
	No. & Name& $p/n$& $\lambda_1(\bR)/p$ & $GK^\bR$& $GK_{p/n,\lambda_1(\bR)/p}$ & $CP^\bR(.7)$ & $p$  \\
	\hline
	1 & Sorlie &\ \ \ 5.4 & .110 & .16228 \  & .18640 \ & .05044 & \ \ \ 456 \\
	2 & Gravier &\ \ 17.3  & .083 & .05370 \  & .05783 \ &.01136 & \ 2905\\
	3 & Alon &\ \ 32.3 &  .450 & .03000 \  &.03100 \ & .00200&\ 2000\\
	4 & Yeoh &\ \ 50.9 & .145& .01956 \  &.01964 \ & .00539 &  12625\\
	5 & Gordon & \ \ 69.2 & .087& .01436 \  & .01444 \ &.00367  &12533\\
	6 & Tian &\ \ 73.0 & .089 & .01362 \ & .01370 \ & .00554 &12625\\
	7 & Shipp &\ \ 92.6  &.213 & .01066 \ & .01080 \ & .00182&\ 7129\\
	8 & Chiaretti & \ \ 98.6 & .181 & .01006  \  & .01014  \ &.00190 &12625\\
	9 & Golub & \ \ 99.0 &.149 & .00996 \  & .01010 \ &.00351 &\ 7129\\
	10 & Pomeroy &\ 118.8  &.266 & .00828 \ & .00842 \ &.00224&\ 7128\\
	11 & West & \ 145.5 &.162 & .00673 \  & .00687 \ & .00224&\ 7129\\
	12 & Burczynski & \ 175.5 &.115 & .00565 \ & .00570 \ & .00193 & 22283\\
	13 & Chin & \ 188.3 &.164 & .00527 \  & .00531 \ & .00158& 22215\\
	14 & Nakayama  & \ 212.2 &.073 & .00467 \ & .00471 \ &.00202 & 22283\\
	15 & Chowdary & \ 214.2 & .699 & .00462 \  & .00467 \ &.00004 &22283\\
	16 & Borovecki &\ 718.8   &.173 & .00135 \ & .00139 \ & .00058 & 22283\\
	\hline
	\caption{DNA microarray datasets.}
	\label{table:micro}
\end{longtable}

Table~\ref{table:micro} shows that $GK^\bR>CP^\bR(.7)$ and $GK^\bR$ is nonincreasing in $p/n$ as Figure~\ref{fig:cpgk1}. Moreover, by Theorem~\ref{thm:fig}~(\ref{assert:all nonzero}), $GK_{p/n,\lambda_1(\bR)/p}$ is  $n/p$ because all datasets have $p/n$ more than $(1/\sqrt{1-\lambda_1(\bR)/p}+1)^2$. Indeed, the empirical values $GK^\bR$ are around $n/p$ for all datasets having $p/n$ sufficiently greater than the threshold $(1/\sqrt{1-\lambda_1(\bR)/p}+1)^2$.

\subsection{Datasets from economics}
We consider the datasets of  returns of $p$ S\&P500 stocks for $n$ trading
days.  Laloux et al.~\cite{laloux} fitted  the histogram of the
correlation matrix $\bR$ of such a dataset to the density function of a scaled Mar\v{c}enko-Pastur distribution.

Table~\ref{tbl:212} is the list of $p/n$, $\lambda_1(\bR)/p$, $GK^{\bR}$, $GK_{p/n,\lambda_1(\bR)/p}$, and $p=212$ S\&P500 stocks of various periods, in the increasing order of $\lambda_1(\bR)/p.$
\begin{table}[ht]
\begin{tabular}{|c|c|c|c|c|c|c|}
\hline
No&Period & $p/n$ & $\lambda_1(\bR)/p$ &  $GK^\bR$& $GK_{p/n,\lambda_1(\bR)/p}$ &$p$ \\
\hline
1&1993-01-04-1995-12-29 	  &.2808 &.1103 &.3302 & \ \, .3747 &212 \\
2&1993-01-04-2022-08-01 	  &.0285 &.3137 &.1038 & 0 \ \ \ &212  \\
3&2012-08-01-2022-08-01	  &.0843 &.4026 &.0991  & 0 \ \ \ &212  \\
4&2005-01-04-2022-08-01	  &.0479 &.4232 &.0849 &0 \ \ \ &212  \\
5&2005-01-04-2013-12-30	  &.0938 &.4503 &.0849 & 0 \ \ \ &212 \\
\hline
\end{tabular}\caption{\label{tbl:212}The returns of S\&P500 datasets.}
\end{table}
$GK^\bR$ is decreasing in $\lambda_1(\bR)/p$ in Table~\ref{tbl:212}, as Observation~\ref{obs}~\eqref{GK decreasing in rho}. It is worth noting that the estimator $GK_{p/n,\lambda_1(\bR)/p}$ is mostly 0.

Next, we consider similar but more categorized datasets. The return of a stock is more correlated with the return of a stock of the
same industry classification sector than with the return of a stock  of a different
industry classification sector.  
Table~\ref{tbl:gics} is the list of \emph{global industry classification
standard} (GICS) sector of S\&P500, $p/n$, $\lambda_1(\bR)/{p}$, $GK^\bR$, $GK_{p/n,\lambda_1(\bR)/p}$,  and $p$. All the periods are 2012-01-04/2022-12-31. Table~\ref{tbl:gics} is ordered in the increasing order of $\lambda_1(\bR)/p$.
All the $p$ of Table~\ref{tbl:gics} are smaller than the $p=212$ of Table~\ref{tbl:212}, but some GICS sectors have larger $\lambda_1(\bR)/p$. In Table~\ref{tbl:gics}, 
$GK^\bR$ is decreasing if we omit 
the first line~(Communication services, $p=19$),
the fourth line~(Consumer Staples, $p=23$), 
the seventh line~(Material, $p=24$), the tenth line~(Energy, $p=16$) from all the  lines.
 Observation~\ref{obs}~\eqref{GK decreasing in rho} for Guttman-Kaiser criterion holds for $p\ge28$. 

\begin{table}[ht]
\begin{tabular}{ |c|c|l|c|c|c|c|}
\hline
No&GICS & $p/n$ & $\lambda_1(\bR)/p$ &  $GK^\bR$& $GK_{p/n,\lambda_1(\bR)/p}$ &$p$ \\
\hline
1& Communication Serv. & .0076 & .3571 & .2105& 0&19\\
2&Consumer Discret.
 & .0207  & .3843 & .1538&0&52 \\  
3&Health Care 	&	.0187& .3948	&	.1489&0&47\\
4&Consumer Staples   &	.0091	& .4302 &	.1739 &0 &23\\
5&Information Tech.  	&	.0246	& .4648 &	.0968&0&62\\
6&Industrials 	&	.0258	& .4985 &	.0923&0&65\\
7&Materials & .0095 & .4990 & .1667 &0& 24\\
8&Real estate & .0119 & .5819 & .1000 & 0& 30\\
9&Financials 	&	.0250	& .6086&	.0794&0&63\\
10&Energy 	&	.0064 & .6872	&	.0625&0&16\\
11&Utilities & .0111	& .6897 & .0714 &0& 28\\
\hline	 	
\end{tabular}
\caption{The returns of S\&P500 stocks per GICS. 
}
\label{tbl:gics}
\end{table}
The estimator $GK_{p/n,\lambda_1(\bR)/p}$ of $GK^\bR$ in Table~\ref{tbl:212} and Table~\ref{tbl:gics} are all 0 for $\lambda_1(\bR)/p>0.110$. The estimator $GK_{p/n,\lambda_1(\bR)/p}=1-F_{p/n,1}\left(1/(1-\lambda_1(\bR)/p)\right)$ is $0$ if $p/n<1$ but $\lambda_1(\bR)/p<1$ is sufficiently large. The first author computed the time series of equi-correlations for the datasets, by employing GJR GARCH~\cite{GJR} with correlation structure being \emph{dynamic equicorrelation} \cite{eng-kel}. Then,  $\lambda_1(\bR)/p$ is always larger than the time averages of the time series of the equi-correlation coefficient.

To detect the correlation structure from stock datasets~(Table~\ref{tbl:212} and Table~\ref{tbl:gics}), we apply a hierarchical clustering algorithm to the sample correlation matrices $\bR$ of the datasets, and then computed their color heat maps. We found that they are very different from the heat map Figure~\ref{heatmapperson} of the correlation matrix of the binary MSA dataset. The correlation matrices of the stock returns have diagonal block structures. Moreover, the sizes of the GICS correlation matrices are small. By these,  $GK^\bR$ of the stock returns datasets may be discrepant from our estimator  $GK_{p/n,\lambda(\bR)/p}$.

Finally, we consider two household datasets in 2019  by area classification, from \cite{stat}. One is for the amount of assets per households and the other is for the average yearly income. These datasets have response variables: the amount of assets and liabilities per household and the average yearly income from the whole of Japan. Meanwhile, the explanatory variables are the amount of assets per households and the average yearly income from each of 66 regions in Japan. 
We can detect the \emph{multicollinearity}~\cite{chatterjee2006regression,J} by \emph{variance inflation factors} (VIFs)~\cite{chatterjee2006regression,J} of the explanatory variables. If all the explanatory variables are uncorrelated, then all the VIFs are equal to 1, but if severe multicollinearities exist, then the VIFs for some explanatory variables become very large~\cite{chatterjee2006regression,J}.  
\begin{table}[ht]
\begin{tabular}{|c|c|c|c|c|c|c|}
\hline
Name& $\min$ VIF& $R^2$ & $\lambda_1(\bR)/p$ & $GK^{\bR}$& $p$& $p/n$\\
\hline
The amount of assets 2019 &8.44$\times10^2$ & 1  &.980 & .015  &66 &.051 \\
The average yearly income 2019&1.47$\times10^4$ & 1  &.993 & .015  &66 &.048 \\
\hline
\end{tabular}
\caption{The household datasets of 2019.\label{tbl:household}
}
\end{table}

Table~\ref{tbl:household} is the list of name, $\min$ VIFs, the (adjusted) coefficient $R^2$ of determination,  $\lambda_1(\bR)/p$, $GK^{\bR}$, $p=66$ variables, and $p/n$, in the increasing order of $\lambda_1(\bR)/p.$  Since $R^2=1$ for both datasets, we can assume that all the explanatory variables have the equi-correlation coefficient $\rho=1$, which means $\lambda_1(\bR)/p$ is approximately 1. One of future work is to discuss the multicollinearity and other correlation structures~\cite{bai-zhou, bryson,eng-kel,fan,david,Pe} among variables  with the extreme eigenvalues and the bulk eigenvalues of datasets, by employing random matrix theory. 

\section{Conclusion}
\label{sec:conclusion}
For assessing the (essential) dimensionality of empirical data, the Guttman-Kaiser criterion is a widely employed criterion. This rule may be the most used for retaining principal components and factors owing to its clarity, ease of implementation~\cite{lean}, and default stopping rule in statistical tools such as SPSS and SAS.
In this paper, we have shown the scaling of Mar\v{c}enko-Pastur distribution when a dataset is from an ENP and $0\le \rho<1$ by the LSD of the sample correlation matrix.
This scaling of Mar\v{c}enko-Pastur distribution explains the ``phase transitions'' of Guttman-Kaiser criterion depending on whether $\rho=0$ or not as $n,p\to\infty,\ p/n\to c>0$. Moreover, by Observation~\ref{obs}, we show the behavior of Guttman-Kaiser criterion where this criterion retains the small number of principal components or factors in small $c$ for $0<\rho<1$ and the limit of $GK^{\bR}$ is 1/2 in $c\downarrow0$ for $\rho=0$. In high-dimensional statistics of various fields, when the number of variables are smaller than the size of a sample, a global correlation among the variables causes a perceptible global impact, even if the correlation is minute. 

\section*{Acknowledgment}
The authors are partially supported by Graduate School of Science, Tohoku University. The MSA dataset is kindly provided from Prof. A. A. Quadeer through the email correspondence with  Prof.  M. R. McKay on 22 September, 2022. 

\appendix
\section{Proofs of Theorem~\ref{thm:EqCor CPMconvCPc} and Theorem~\ref{thm:null CPMconvCPc}}
Suppose that $\M$ is a positive real symmetric semi-definite matrix of order $p$ where all $p$ eigenvalues follow~\eqref{eigenvalues1}.
For $k>0$, the eigenvalues of matrix $k\M$ are~\eqref{eigenvalues1} multiplied by $k$ because the eigenvalues of matrix $\M$ follows~\eqref{eigenvalues1}. Thus, for a random distribution function $G^\M$ defined by \eqref{def:GT}, $G^{k\M}(kx)=G^\M(x)$ for any $x\in\R$. 
We represent $G^\M$ by using $x\vee 0:=\max(x,0)$ and  the probability measure $\meas^\M$ corresponding to the distribution function $F^\M$. 
\begin{lem}\label{lem:GM}
  \begin{align*}
G^\M(x)&=\frac{\int_{[0,x\vee0]}\lambda\, d\meas^\M(\lambda)}{\int_{\R}\lambda\, d\meas^\M(\lambda)},
\qquad(x\in\R).
\end{align*}
\end{lem}
\begin{proof}
By~\eqref{eigenvalues1}, $\meas^\M\left( (-\infty,0)\right)=0$. Because $F^\M(x)=p^{-1}\sum_{i=1}^p \I_{\lambda_i\le x}$ for all $x\in\R$, $G^\M(x)=\sum_ {\lambda_i\le x}\lambda_i/\sum_{i=1}^p\lambda_i$. Since ${\sum_ {\lambda_i\le x} \lambda_i}\le {\sum_ {\lambda_i\le z} \lambda_i}$ for $x\le z$, $G^\M$ is nondecreasing. Since $[\lambda_i,\, \infty)$ has a left endpoint, the right-continuity is clear from $G^\M$. Moreover, $\lim_{x\to -\infty}G^\M(x)=0$ and $\lim_{x\to +\infty}G^\M(x)=1$. Therefore $G^\M$ is indeed a random distribution function.
\end{proof}
\begin{lem}\label{lem:finiteGm}
For $u\in(0,1)$, $\min\set{\lambda_i | \lambda_i>0,\ 1\le i\le p}\le \qtl{G^\M}(u)\le \lambda_1$.
\end{lem}
\begin{proof}
By \eqref{def:GT}, 
\begin{align}\qtl{G^\M}(u)
=\inf\Set{x| \sum_{\lambda_i> x} \lambda_i \big/ \sum_ {i=1}^p \lambda_i\le 1-u}.\label{qtl:GM}\end{align}
Because $\lambda_1$ is the largest eigenvalue of $\M$,
$\sum_{\lambda_i>\lambda_1}\lambda_i=0$. 
By $0<u<1$, $\lambda_1$ is an element of the set in the right side of \eqref{qtl:GM}. 
Hence, $\left(G^\M\right)^-(u)\le\lambda_1$. 
On the other hand, if $x=0$ and $x< \min_{1\le i\le p}\lambda_i$, $\sum_{\lambda_i> x} \lambda_i/\sum_ {i=1}^p \lambda_i=1$. Thus, by $u\in(0,1)$, whenever $\sum_{\lambda_i> x} \lambda_i \big/ \sum_ {i=1}^p \lambda_i\le 1-u$, $x \ge\min\set{\lambda_i | \lambda_i>0,\ 1\le i\le p}$. Therefore, $\min\set{\lambda_i | \lambda_i>0,\ 1\le i\le p}\le\left(G^\M\right)^-(u)\le\lambda_1$. 
\end{proof}

In sequel, for $t\in(0,\,1)$, we  show that $\qtl{G^\M}(1-t)$ is a threshold for eigenvalues CPV rule retains.
In Definition~\ref{def:cp}, we gave the representation of  $CP_{c,\rho}$ with the complementary distribution function of $F_{c,1}$ and the quantile function of $G_{c,1}$.
Like this, but by replacing the subscripts with the superscripts $\M$ and by replacing $(1-t)/(1-\rho)$ with $1-t$,
we  have 
the following representation of $CP^{\M}(t)$:
\begin{lem}\label{lem:CPM}For any real symmetric positive semi-definite matrix $\M$,   $$CP^{\M}(t)=\overline{F^\M}(\qtl{G^\M}(1-t)),\qquad(0<t<1).$$
\end{lem}
\begin{proof} 
Let $q=pCP^{\M}(t)$.
By $\sum_{\lambda_i>\lambda_{q+1}} \lambda_i\big/\sum_{i=1}^p\lambda_i\le t$ and $\lambda_q>\lambda_{q+1}$, \eqref{qtl:GM} implies
$\lambda_{q+1}\ge \qtl{G^\M}(1-t)$.
If $x<\lambda_{q+1}$, then $\sum_{\lambda_i>x} \lambda_i\big/\sum_{i=1}^p\lambda_i> t$ by \eqref{qtl:GM}.
Thus, $\lambda_{q+1}\le \qtl{G^\M}(1-t).$
As a result, $\lambda_{q+1}= \qtl{G^\M}(1-t).$
Thus, by Definition~\ref{def:CPM},
$CP^{\M}(t)=\frac{1}{p}\times \#
\Set{i| \lambda_i> \left(G^\M\right)^-(1-t)}=\overline{F^\M}(\qtl{G^\M}(1-t)).$
\end{proof}

We can readily observe the following:
\begin{lem}\label{lem:Gc}
$G_{c,\sigma^2}$ is a continuous distribution function which is strictly increasing on $(a_{\sigma^2}(c), b_{\sigma^2}(c))$,
and \begin{align*}\G{c}{\sigma^2}(x)=\frac{\int_{[0,\,x\vee0]}\lambda d\MPmeas{c}{\sigma^2}(\lambda)}{\int_{\R}\lambda d\MPmeas{c}{\sigma^2}(\lambda)},\qquad(x\in\R).\end{align*}
Moreover, $G_{c,\sigma^2}(x)=G_{c,1}(x/\sigma^2)$ for all $x\in \R$.
\end{lem}

\begin{lem}\label{LEMMA3:Cp} 
The function $\G{c}{\sigma^2}$ restricted to an open interval $J=(a_{\sigma^2}(c),\,b_{\sigma^2}(c))$ has
$\qtl{\G{c}{\sigma^2}}$ as the inverse function.
Moreover, $\qtl{\G{c}{\sigma^2}}$ is a continuous, strictly increasing function from the unit open interval $(0,1)$ to $J$.
\end{lem}
\begin{proof}
By Lemma~\ref{lem:Gc}, $G_{c,\sigma^2}$ is strictly increasing on $J$.
Let $\tilde{x}\in J$. Then
$\qtl{\G{c}{\sigma^2}}(G_{c,\sigma^2}(\tilde{x}))=\tilde{x}$ by $\qtl{\G{c}{\sigma^2}}(G_{c,\sigma^2}(\tilde x))=\inf\Set{x|G_{c,\sigma^2}(x)\ge G_{c,\sigma^2}(\tilde{x})}$. 
For any $t\in(0,1)$,   $\qtl{\G{c}{\sigma^2}}(t)=\inf\Set{x|G_{c,\sigma^2}(x)=t}$. Since $G_{c,\sigma^2}$ is continuous function, $G_{c,\sigma^2}(\qtl{\G{c}{\sigma^2}}(t))=t$.
Hence, $\qtl{\G{c}{\sigma^2}}$ is the inverse function of $G_{c,\sigma^2}$ restricted to $J$, and $\qtl{\G{c}{\sigma^2}}$ is a continuous, strictly increasing function from $(0,1)$ to $J$.\end{proof}

\begin{lem}\label{lem:limQcp}
For $c,\sigma^2,k>0$, $z\in\R$,
\begin{enumerate}\rm
    \item 
$\qtl{k\G{c}{\sigma^2}}\left(z\right)=\qtl{\G{c}{\sigma^2}}\left(\fra{z}{k}\right).$
\item $\overline{F_{c,\sigma^2}}\left(\qtl{\G{c}{\sigma^2}}\left(z\right)\right)=\overline{F_{c,1}}\left(\qtl{\G{c}{1}}\left(z\right)\right).$

\item\label{assert:cpv in unit interval} $\overline{F_{c,1}}\left(\qtl{\G{c}{1}}(t)\right)\in(0,\,1)$, if $t\in(0,\,1)$.\end{enumerate}
\end{lem}
\begin{proof} (1) ${\qtl{k\G{c}{\sigma^2}}(z)}=\inf\Set{x\in\R|kG_{c,\sigma^2}(x)\ge z}$. 
By $k>0$, it is equal to
    $\inf\Set{x\in\R|G_{c,\sigma^2}(x)\ge z/k}=\qtl{\G{c}{\sigma^2}}\left(z/k\right).$

\medskip\noindent
(2) $\overline{F_{c,\sigma^2}}\left(\qtl{\G{c}{\sigma^2}}\left(z\right)\right)=\overline{F_{c,1}}\left(\frac{1}{\sigma^2}{\qtl{\G{c}{\sigma^2}}\left(z\right)}\right)$. By the definition of the quantile function, $\qtl{\G{c}{\sigma^2}}\left(z\right)/\sigma^2=\inf\Set{x|G_{c,\sigma^2}(\sigma^2x)\ge z}$ which is equal to $\qtl{\G{c}{1}}\left(z\right)$.
Thus, $\overline{F_{c,\sigma^2}}\left(\qtl{\G{c}{\sigma^2}}(z)\right)=\overline{F_{c,1}}\left(\qtl{\G{c}{1}}(z)\right)$.

\medskip\noindent
\eqref{assert:cpv in unit interval} By Lemma~\ref{LEMMA3:Cp}, we have $\qtl{\G{c}{1}}:(0,1)\to (a_{1}(c),b_{1}(c))$. 
Because
$\overline{F_{c,1}}$ restricted to $(a_1(c),\, b_1(c))$ is a strictly decreasing function to $(0,\,1)$, it follows that  $\overline{F_{c,1}}\left(\qtl{\G{c}{1}}(t)\right)\in (0,1)$ for $t\in(0,\,1)$.
\end{proof}

The following proposition is \emph{Arzel\`{a}'s dominated convergence theorem}~\cite{arzel}.
By this, we prove that the distribution function $G^\bR$ ($G^\S $, resp.) almost surely converges pointwise to a \emph{defective} distribution function $(1-\rho)G_{c,1-\rho}$ ($(1-\rho)G_{c,\sigma^2(1-\rho)}$, resp.).
\begin{prop}[\protect{\cite[{Theorem A}]{arzel}}]
Let $\{f_n\}$ be a sequence of Riemann-integrable functions defined on a bounded and closed interval $[a,\,b]$, which converges on $[a,\,b]$ to a Riemann-integrable function $f$. If there exists a constant $M>0$ satisfying $|f_n(x)|\le M$ for all $x\in [a,\,b]$ and for all $n$, then $\lim_{n\to\infty} \int_a^b\left|f_n(x)-f(x)\right|dx=0$. In particular,
$$\lim_{n\to\infty}\int_a^bf_n(x)dx=\int_a^b\lim_{n\to\infty}f_n(x)dx=\int_a^bf(x)dx.$$
\end{prop}
\begin{lem}\label{lem:GMconvGc}
Suppose $X_1,\ldots,X_n\stackrel{\mbox{\rm i.\@i.\@d.\@}}{\sim} \mathrm N(\bm{\mu},\,\mathbf D\Rc(\rho)\mathbf D)$ for $0\le \rho<1$. Suppose $n,p\to\infty$ with $p/n\to c>0$. Then, it holds almost surely that for any $x\in\R$,
\begin{enumerate}\rm
    \item\label{assert:GRP limit} $G^\bR(x)\to (1-\rho)\G{c}{1-\rho}(x)$, 
    \item\label{assert:GSP limit} $G^\S (x)\to (1-\rho)\G{c}{\sigma^2(1-\rho)}(x)$ for $\bm{\mu}=\bm{0}$ and $\mathbf D=\sigma\mathbf I$ with $\sigma>0$.
\end{enumerate}
\end{lem}
\begin{proof}
For a symmetric positive semi-definite $\M$, we know that $\int_{[0,x\vee0]}\lambda d\meas^\M(\lambda)=\int_0^{x\vee0}\meas^\M\left(\set{\lambda| f(\lambda)>h}\right) dh$ for $f:\R\to\R$  such that $f(\lambda)$ be $\lambda$ for $0<\lambda\le x$ and 0 otherwise. 
For $h\ge0$, $f(\lambda)>h$ if and only if $h<\lambda\le x$. 
Then, $\meas^\M\left(\set{\lambda|f(\lambda)>h}\right)$ $=F^\M(x)-F^\M(h)$ for $0\le h<x$, and 0 otherwise.
Similarly, $\MPmeas{c}{\sigma^2}(\set{\lambda|f(\lambda)>h})$ $=F_{c,\sigma^2}(x)-F_{c,\sigma^2}(h)$ for $0\le h<x$, and 0 otherwise. 

\bigskip\noindent\eqref{assert:GRP limit} 
Since the sample correlation matrices $\bR$ are invariant under scaling  of variables, we can write $\mathbf{D}=\mathbf{I}$ without loss of generality. Moreover, we can assume that $\bm \mu=\bm 0$ because $\bR$ is invariant under sifting.
Because $\int_{\R}\lambda\, d\meas^{\bR}(\lambda)=\Tr \bR/p=1$ and $\int_{\R}\lambda d\MPmeas{c}{1-\rho}(\lambda)=1-\rho$, the denominator of $G^\bR(x)$ in Lemma~\ref{lem:GM}~$(G_{c,1-\rho}(x)$ in Lemma~\ref{lem:Gc}, resp.) is 1~($1-\rho$, resp.) for all $x\in \R$. By this, the representation of Lebesgue integral with a Riemann integral, Lemma~\ref{lem:GM} and Lemma~\ref{lem:Gc},  for all positive integer $p$ and $x\in\R$, 
\begin{align}
&\left|G^\bR(x)-(1-\rho)\G{c}{1-\rho}(x)\right|=
\left|\int_{[0,x\vee0]}\lambda d\meas^{\bR}(\lambda)- \int_{[0,x\vee0]}\lambda d\meas_{c,1-\rho}(\lambda)\right|\nonumber\\
&=\left|\int_0^{x\vee0} (F^{\bR}(x) - F^{\bR}(h)) - (F_{c,1-\rho}(x) - F_{c,1-\rho}(h)) dh\right|\nonumber\\
&\le|x\vee0||F^{\bR}(x)-F_{c,1-\rho}(x)|+\left|\int_0^{x\vee0}F^{\bR}(h)-F_{c,1-\rho}(h)dh\right|\label{ineq:a}.
\end{align}
By Theorem~\ref{thm:equilsd2}, for all $c>0$, it holds almost surely that for any $x\ne0$ we have $F^{\bR}(x)\to F_{c,1-\rho}(x)$.  
Thus, for any $c>0$, it holds almost surely that for any $x$, the first term of the right side of \eqref{ineq:a} converges to 0.
Define two functions $\tilde F^{\bR}:[0,\, x]\to [0,\, 1]$ and  $\tilde{F}_{c,1-\rho}:[0,\,x]\to [0,\, 1]$ as follows:
$\tilde F^{\bR}(h) = 0\ (h=0)$; $  F^{\bR}(h)\ (0<h\le x)$; and $\tilde{F}_{c,1-\rho}(h) = 0\  (h=0)$; $F_{c,1-\rho}(h)\ (0<h\le x)$.
The four functions $\tilde  F^{\bR}$, $ F^{\bR}$, $\tilde F_{c,1-\rho}$, and $F_{c,1-\rho}$ are Riemann-integrable over $[0,x]$, because they are nondecreasing.
Moreover,
$\int_0^{x\vee0} \tilde F^{\bR}(h) dh = \int_0^{x\vee0}  F^{\bR}(h) dh$, and $\int_0^{x\vee0} \tilde{F}_{c,1-\rho}(h) dh = \int_0^{x\vee0} F_{c,1-\rho}(h) dh$.
For any $h\in[0,x]$, $|\tilde F^{\bR}(h)|\leq 1$. 
By Theorem~\ref{thm:equilsd2}, for all $c>0$, it holds almost surely that $\tilde F^{\bR}$ converges pointwise to $\tilde F_{c,1-\rho}$. 
Hence, by  Arzel\`{a}'s dominated convergence theorem~\cite{arzel}, it holds almost surely that
$$\int_0^{x\vee0}  F^{\bR}(h) dh=\int_0^{x\vee0} \tilde  F^{\bR}(h) dh\to \int_0^{x\vee0} \tilde F_{c,1-\rho} (h) dh=\int_0^{x\vee0} F_{c,1-\rho} (h) dh.$$
Thus, the second term of the right side of \eqref{ineq:a} converges almost surely to 0.
 Therefore, for all $c>0$,  almost surely 
 $G^\bR$ converges pointwise to $(1-\rho)G_{c,1-\rho}$.

\bigskip
\bigskip\noindent\eqref{assert:GSP limit} 
By a similar argument to~\eqref{ineq:a}, $
\left|\int_{[0,x\vee0]}\lambda d\meas^\S (\lambda)-\int_{[0,x\vee0]}\lambda d\meas_{c,\sigma^2(1-\rho)}(\lambda)\right|$ is \begin{align}
&=\left|(x\vee0) (F^\S(x) - F_{c,\sigma^2(1-\rho)}(x)) + \int_0^{x\vee0}(F_{c,\sigma^2(1-\rho)}(h) - F^\S(h)) dh\right|\nonumber\\
& \le|x\vee0||F^\S(x)-F_{c,\sigma^2(1-\rho)}(x)|+\left|\int_0^{x\vee0}F^\S(h)-F_{c,\sigma^2(1-\rho)}(h)dh\right|.\label{ineq:b}
\end{align}
By Theorem~\ref{thm:equilsd},
for all $c>0$, almost surely, for any $x\in\R$, the first term of \eqref{ineq:b} converges to 0.
Define two functions $\tilde{F}^{\S}:[0,\, x]\to [0,\, 1]$ and  $\tilde{F}_{c,\sigma^2(1-\rho)}:[0,\,x]\to [0,\, 1]$ as follows:
$\tilde{F}^{\S}(h) = 0\ (h=0)$; $ F^\S(h)\ (0<h\le x)$; and $\tilde{F}_{c,\sigma^2(1-\rho)}(h) = 0\  (h=0)$; $F_{c,\sigma^2(1-\rho)}(h)\ (0<h\le x)$.
Note that $\tilde F^\S$, $F^\S$, $\tilde F_{c,\sigma^2(1-\rho)}$, and $F_{c,\sigma^2(1-\rho)}$ are Riemann-integrable over $[0,x\vee0]$, because they are nondecreasing.
Moreover,
$\int_0^{x\vee0} \tilde{F}^\S(h) dh = \int_0^{x\vee0} F^\S(h) dh$, and $\int_0^{x\vee0} \tilde{F}_{c,\sigma^2(1-\rho)}(h) dh = \int_0^{x\vee0} F_{c,\sigma^2(1-\rho)}(h) dh$.
For any $h\in[0,x]$, $|\tilde{F}^{\S}(h)|\leq 1$. 
 By Theorem~\ref{thm:equilsd},  for all $c>0$,
 almost surely, 
 $\tilde F^\S$  converges pointwise to 
 $\tilde F_{c,\sigma^2(1-\rho)}$. 
Hence, by  Arzel\`{a}'s dominated convergence theorem~\cite{arzel},
$$\int_0^{x\vee0} F^\S(h) dh=\int_0^{x\vee0} \tilde F^\S(h) dh$$
converges almost surely to 
$$\int_0^{x\vee0} \tilde F_{c,\sigma^2(1-\rho)} (h) dh=\int_0^{x\vee0} F_{c,\sigma^2(1-\rho)} (h) dh.$$
 As a result, the second term of the right side of \eqref{ineq:b} converge almost surely to 0.
Thus, for all $c>0$,  almost surely  $$\int_{[0,x\vee0]}\lambda d\meas^\S (\lambda)\to\int_{[0,x\vee0]}\lambda d\meas_{c,\sigma^2(1-\rho)}(\lambda),\qquad (x\in\R).$$
Besides, $\int_{\R}\lambda d\meas^\S (\lambda) \to \sigma^2$ by Lemma~\ref{lem:lim2}. Therefore, for all $c>0$, almost surely, for all $x\in\R$,
\begin{align*}G^\S (x)=\frac{\int_{[0,x\vee0]}\lambda d\meas^\S (\lambda)}{\int_{\R}\lambda d\meas^\S (\lambda)} \to \frac{\int_{[0,x\vee0]}\lambda d\meas_{c,\sigma^2(1-\rho)}(\lambda)}{\sigma^2}=(1-\rho)\G{c}{\sigma^2(1-\rho)}(x),
\end{align*}
as desired.
\end{proof}
\begin{prop}
[\cite{EH}]\label{lem:generalized inverse}
Suppose $f:\R\to\R$ is nondecreasing, and let $f(-\infty)=\lim_{x\downarrow-\infty}f(x)$ and $f(\infty)=\lim_{x\uparrow \infty}f(x)$. Let $x,y\in\R$. Then, $f(x)\ge y$ implies $x\ge f^-(y).$ The other implication holds if $f$ is right-continuous. Furthermore, $f(x)< y$ implies $x\le f^-(y).$
\end{prop}
\begin{lem}\label{lem:Parzen}
If a distribution function $H_n$ weakly converges to a defective distribution function $H$, then $H_n^-(t)$ converges to $H^-(t)$ for any continuity point $t$ of $H^-$. 
\end{lem}
\begin{proof}
This lemma is Theorem~2A~\cite{Parzen80} when $H$ is not defective.
We  use Theorem~2A~\cite{Parzen80} to prove this lemma. 
Let $0<t<1$ be a continuity point of $H^-$. Then one can pick a sequence $\epsilon_k$ that converges to 0 such that $H^-(t)-\epsilon_k$ and $H^-(t)+\epsilon_k$ are continuity points of $H$ and  
$H(H^-(t) - \epsilon_k) < t < H(H^-(t) + \epsilon_k), $
for each fixed $k$. By the assumption, $N_k$ may be chosen such that  
$H_n(H^-(t) - \epsilon_k ) < t < H_n(H^-(t) + \epsilon_k ),$
for any $n>N_k$. Therefore, by Proposition~\ref{lem:generalized inverse}, for $n \ge  N_k$,
$H^-(t)-\epsilon_k\le H_n^-(t)\le H^-(t)+\epsilon_k.$
$H_n^-(t)\to H^-(t)$ can be deduced.
\end{proof}

\begin{lem}\label{LEMMA6:Cp}
Suppose $X_1,\ldots,X_n\stackrel{\mbox{\rm i.\@i.\@d.\@}}{\sim} \mathrm N(\bm{\mu},\,\mathbf D\Rc(\rho)\mathbf D)$ for $0\le \rho<1$. Suppose $n,p\to\infty$ with $p/n\to c>0$. Then, it holds almost surely that for any $u\in(0,\, 1-\rho]$,
\begin{enumerate}\rm
    \item $\qtl{G^\bR}(u)\to \qtl{(1-\rho)\G{c}{1-\rho}}(u)$, and
     \item $\qtl{G^\S }(u)\to \qtl{(1-\rho)\G{c}{\sigma^2(1-\rho)}}(u)$ for $\bm{\mu}=\bm{0}$ and $\mathbf D=\sigma\mathbf I$ with $\sigma>0$.
\end{enumerate}
\end{lem}

\begin{proof}
By Lemma~\ref{LEMMA3:Cp}, both of $\left((1-\rho)G_{c,1-\rho}\right)^-$ and $\left((1-\rho)G_{c,\sigma^2(1-\rho)}\right)^-$ are continuous at $u\in(0,\,1-\rho].$ Therefore, the conclusion follows from
Lemma~\ref{lem:GMconvGc} and Lemma~\ref{lem:Parzen}.
\end{proof}
\noindent\textbf{Proof of Theorem~\ref{thm:EqCor CPMconvCPc}}. 
Let $\M$ be $\bR$ or $\S$ and let $\lambda_1\ge\cdots\ge\lambda_p$ be the eigenvalues of $\M$.

(1) By Lemma~\ref{lem:finiteGm}, $\qtl{G^\bR}\left({1-t}\right)>0$. Therefore, $\qtl{G^\bR}\left({1-t}\right)$ is a continuity point of $F_{c,1-\rho}$ for any $c>0$.
Thus, by Lemma~\ref{lem:CPM} and Theorem~\ref{thm:equilsd2}, for all $c>0$, it holds almost surely that for any $t\in[\rho,\,1)$,
\begin{align}\begin{aligned}\label{pqr}
&CP^\bR(t) - \overline{F_{c,1-\rho}}\left(\qtl{G^\bR}\left({1-t}\right)\right)\\
&=
\overline{F^{\bR}}\left(\qtl{G^\bR}\left({1-t}\right)\right) - \overline{F_{c,1-\rho}}\left(\qtl{G^\bR}\left({1-t}\right)\right)\to0.
\end{aligned}
\end{align}
By Lemma~\ref{LEMMA6:Cp} and Lemma~\ref{lem:limQcp},
we have $\qtl{G^\bR}(1-t)\cas\qtl{(1-\rho)\G{c}{1-\rho}}$ $(1-t)=\qtl{\G{c}{1-\rho}}\left(s\right)$, where $s=({1-t})/({1-\rho}).$
By the premise $0<\rho\le t<1$, $0<s\le1$,   $\qtl{\G{c}{1-\rho}}\left(s\right)=\inf\set{x | \G{c}{1-\rho}(x)\ge s}$ is positive and finite.
Therefore, $\qtl{\G{c}{1-\rho}}\left(s\right)$ is a continuous point of $\overline{F_{c,1-\rho}}.$
By continuous mapping theorem~\cite[Theorem~2.3]{Van98},
it holds almost surely that for any $t\in[\rho,\,1)$,
$$\overline{F_{c,1-\rho}}\left(\qtl{G^\bR}\left({1-t}\right)\right)\to \overline{F_{c,1-\rho}}\left(\qtl{\G{c}{1-\rho}}\left(s\right)\right).$$
By this and \eqref{pqr}, it holds almost surely that for any $t\in[\rho,1)$, $CP^\bR(t)\to \overline{F_{c,1-\rho}}$ $\left(\qtl{\G{c}{1-\rho}}\left(s\right)\right)= \overline{F_{c,1}}\left(\qtl{\G{c}{1}}\left(s\right)\right)=CP_{c,\rho}(t)$.

(2) We  prove that, it holds almost surely that for any $t\in[\rho,1)$, $CP^\S (t)\to \overline{F_{c,\sigma^2(1-\rho)}}\left(\qtl{\G{c}{\sigma^2(1-\rho)}}\left(s\right)\right)$. 
It can be similarly proved as the convergence of $CP^\bR(t)$ by using Theorem~\ref{thm:equilsd} instead of Theorem~\ref{thm:equilsd2}. By Lemma~\ref{lem:limQcp}, $$\overline{F_{c,\sigma^2(1-\rho)}}\left(\qtl{\G{c}{\sigma^2(1-\rho)}}\left(s\right)\right)=\overline{F_{c,1}}\left(\qtl{\G{c}{1}}\left(s\right)\right).$$ 
Thus, it holds almost surely that for any $t\in[\rho,1)$, $CP^\S (t)\to CP_{c,\rho}(t)$.

\medskip

\noindent\textbf{Proof of Theorem~\ref{thm:null CPMconvCPc}}. The proof for the assertion of $\bR$~($\S$, resp.) is that of Theorem~\ref{thm:EqCor CPMconvCPc} except that $\rho$ should be 0, $t\in [\rho,1)$ should be $t\in (0,1)$, and Theorem~\ref{thm:equilsd2} should be Proposition~\ref{prop:ji3}~(Proposition~\ref{prop:MP LT}, resp.). 

\vspace{2cc}


\end{document}